\documentclass[12pt, a4paper,reqno]{amsart}
\usepackage[a4paper, margin=3cm]{geometry}
\usepackage{amssymb,amsfonts}

\usepackage{color}
\newtheorem{Theorem}{Theorem}[section]

\newtheorem{Lemma}[Theorem]{Lemma}

\newtheorem{Proposition}[Theorem]{Proposition}
\theoremstyle{definition}

\newtheorem{Remark}[Theorem]{Remark}

\renewcommand{\leq}{\leqslant}
\renewcommand{\geq}{\geqslant}

\def\C{\mathbb C}

\def\N{\mathbb N}

\def\R{\mathbb R}

\def\Z{\mathbb Z}
\newcommand{\CF}{\mathcal{F}}

\newcommand{\CH}{\mathcal{H}}

\newcommand{\CJ}{\mathcal{J}}

\newcommand{\CS}{\mathcal{S}}

\newcommand{\CW}{\mathcal{W}}

\def\be{\begin{equation}}
\def\ee{\end{equation}}
\def\bt{\begin{Theorem}}
\def\et{\end{Theorem}}
\def\bi{\begin{itemize}}
\def\ei{\end{itemize}}
\def\bea{\begin{eqnarray}}
\def\eea{\end{eqnarray}}
\def\beast{\begin{eqnarray*}}
\def\eeast{\end{eqnarray*}}
\def\ben{\begin{enumerate}}
\def\een{\end{enumerate}}

\def\bi{\bibitem}

\newcommand{\norm}[1]{\left\Vert#1\right\Vert}
\newcommand{\half}{{\frac{1}{2}}}
\renewcommand{\MR}[1]{} 
\DeclareMathOperator{\dom}{\mathfrak{D}}
\DeclareMathOperator{\fhalf}{\frac{1}{2} }
\newcommand{\abs}[1]{\left\vert#1\right\vert}

\allowdisplaybreaks[4]

\begin{document}
\title[\lowercase{\textit{q}}-deformed  Araki-Woods von Neumann algebras]{On the factoriality of $q$-deformed  Araki-Woods von Neumann algebras}

\author[Bikram]{Panchugopal Bikram}
\author[Mukherjee]{Kunal Mukherjee}
\author[Ricard]{\'{E}ric Ricard}
\author[Wang]{Simeng Wang}
%\date{\today}

\address{School of Mathematical Sciences,
National Institute of Science Education and Research,  Bhubaneswar, An OCC of Homi Bhabha National Institute,  Jatni- 752050, India}
\address{Department of Mathematics, IIT Madras, Chennai - 600036, India.}
\address{Normandie Univ, UNICAEN, CNRS, Laboratoire de Math{\'e}matiques Nicolas Oresme, 14000 Caen, France}
\address{Institute for Advanced Study in Mathematics, Harbin Institute of Technology, Harbin, 150001, China}
\email{bikram@niser.ac.in, kunal@iitm.ac.in,} \email{eric.ricard@unicaen.fr, simeng.wang@hit.edu.cn}

\keywords{ $q$-commutation relations, von Neumann algebras}
\subjclass[2010]{Primary  46L10; Secondary 46L65, 46L55.}

\begin{abstract} 
The $q$-deformed Araki-Woods von Neumann algebras $\Gamma_q(\CH_\R,U_t)^{\prime\prime}$ are factors for all $q\in (-1,1)$ whenever $dim(\CH_\R)\geq 3$. When $dim(\CH_\R)=2$ they are factors as well for all $q$ so long as the parameter defining $(U_t)$ is `small' or $1$ $($trivial$)$ as the case may be. 
\end{abstract} 
\maketitle

\section{Introduction}

To any strongly continuous orthogonal representation $(U_t)$ of $\R$ on a real Hilbert space $\CH_\R$ with $dim(\CH_\R)\geq 2$, Hiai in \cite{Hiai} constructed the $q$-deformed Araki-Woods von Neumann algebras $($hereafter abbreviated as $q$-Araki-Woods algebras$)$ for $-1< q< 1$. These are $W^{\ast}$-algebras usually arising from non-tracial representations of the $q$-commutation relations (a Yang-Baxter deformation of the canonical commutation relations), thereby yielding an interpolation between the Bosonic and Fermionic statistics. Hiai's functor is a fusion of Shlyakhtenko's free CAR functor \cite{Shlyakhtenko} $($associated $W^*$-algebras are called free Araki-Woods factors$)$ and the $q$-Gaussian functor of Bo$\overset{.}{\text{z}}$ejko-Speicher $($associated $W^{*}$-algebras are called Bo$\overset{.}{\text{z}}$ejko-Speicher factors$)$ $($see \cite{BS}$)$. All of these constructions are generalizations of  Voiculescu's $C^*$-free Gaussian functor, which is the central object of study in free probability  \cite{DVN}. Note that when $q=0$, i.e. the Yang-Baxter deformation is trivial $($free case$)$, Hiai's functor reduces to Shlyakhtenko's functor. However, when $(U_t)$ is trivial, Hiai's functor reduces to the $q$-Gaussian functor. Further, when $q=0$ and $(U_t)$ is trivial, one obtains Voiculescu's functor. 

The $q$-Araki-Woods algebras are quite complicated objects. Structural properties of the Bo$\overset{.}{\text{z}}$ejko-Speicher factors have been studied in \cite{avsec,BS,BKS,Dab14,ER,nou,sniady,Shlyakhtenko2,Shlyakhtenko3}, and those of the free Araki-Woods factors have been studied in \cite{Ho09, HR11,BHV15}; though, this list is by no means complete. On the contrary, very little is known about the $q$-Araki-Woods algebras $($see \cite{nou,Nou-Araki,Was}$)$. In fact, the simplest question regarding its factoriality is not known in full generality. This is because, unlike the case when $q=0$, there is very little room to perform meaningful calculations with the standard generators of these algebras owing to the complicated nature of the scalar product $($of the GNS space$)$ and the interference of the modular group. The case when $dim(\CH_\R)=2$ is the hardest and most notorious.  

This paper attempts the factoriality problem of the $q$-Araki-Woods algebras. Before describing our results, though, we note the past efforts in this direction. 

The factoriality of the Bo$\overset{.}{\text{z}}$ejko-Speicher algebras was not a single-handed attempt. In \cite{BKS}, the factoriality was established by Bo$\overset{.}{\text{z}}$ejko, K\"{u}mmerer and Speicher when $\dim(\CH_\R)$ is infinite. By making careful estimates of norms of certain operators on the $q$-deformed full Fock space, \'{S}niady established the factoriality when $dim(\CH_\R)$ is finite but greater than a constant depending on $q$. It was finally settled in \cite{ER} by showing that any standard generator of the Bo$\overset{.}{\text{z}}$ejko-Speicher algebra generates a strongly mixing $($see \cite{CFM2} for defn.$)$ MASA in the ambient algebra. Therefore, the center of the algebra gets arrested in two orthogonal $($with respect to the vacuum state$)$ MASAs and is thus reduced to scalars.  By using freeness and modular data, Shlyakhtenko established the factoriality of the free Araki-Woods algebras in \cite{Shlyakhtenko}. 

Hiai established the factoriality of the $q$-Araki-Woods algebras in \cite{Hiai} in the case when the dimension of the almost periodic part of $(U_t)$ is infinite or  $(U_t)$ is weakly mixing, by showing that the centralizer of the vacuum state has trivial relative commutant $($irreducible inclusion$)$. Unfortunately, there is a gap in the proof of \cite[Thm. 3.2]{Hiai}. To be precise, Hiai's proof holds only in the case when the set of eigenvalues of the analytic generator of $(U_t)$ has a limit point in $\R$ other than $0$. Without this assumption, the conclusion `$\varphi(y^{*}x)=0$' in the last equation in \cite[Thm. 3.2]{Hiai} would fail, and hence, the final statement cannot be concluded. Using the theory of free monotone transport $($see \cite[Thm. 4.5]{Ne15}$)$, Nelson proved that when $\CH_{\R}$ is finite-dimensional, the $q$-Araki-Woods algebras are isomorphic to the free Araki-Woods factors for sufficiently small values of $\abs{q}$; in particular, they are factors, in this case. 

\cite{bik-kunal} closely followed \cite{ER}. In \cite{bik-kunal}, the authors observe that a non-zero vector $\xi\in\CH_\R$ fixed by $(U_t)$ enables the construction of an orthonormal basis of analytic vectors $($of the GNS space$)$ which behaves well as long as one only considers its interaction with elements of the algebra $vN(s_q(\xi))$, where $s_q(\xi)$ is the standard self-adjoint generator of the $q$-Araki-Woods algebra corresponding to $\xi$. This allows exploiting the ideas in \cite{ER} to show that $vN(s_q(\xi))$ is a strongly mixing MASA living inside the centralizer of the vacuum state. Using this, the factoriality of the $q$-Araki-Woods algebras was shown to be true if $(U_t)$ is not ergodic or has a non-zero weakly mixing component. Irreducibility of the centralizer was also obtained when $(U_t)$ is not ergodic, and the almost periodic component is at least two-dimensional. {At the same time, a proof of the same statement on factoriality was obtained independently by Skalski and Wang \cite{SW}; note that the phenomenon of mixing is implicit in this proof too}.

When $(U_t)$ is ergodic and almost periodic, the above ideas involving MASAs freeze. This is because the standard self-adjoint generators no longer generate abelian von Neumann subalgebras with appropriate conditional expectations or operator-valued weights. In fact, it is far worse that these algebras are quasi split \cite{BM2} and therefore admit very large relative commutants. 

Given the previous attempts, we are left to deal with the case when $\CH_\R= \mathbb{R}^2\oplus \mathcal{K}_\R$, where $\mathcal{K}_\R$ is a real Hilbert space $($could be $0)$, and $\mathbb{R}^2$ is reducing subspace for $(U_t)$ with associated sub representation being ergodic. As discussed before, we cannot rely on MASAs anymore, but our strategy is to leverage with a bit of `mixing'. The analysis is split into two cases.

When $\mathcal{K}_\R=0$, denote $\mathcal{A}$ to be the algebra generated by the centralizer and the commutant of the ambient algebra. The novelty of this approach is to identify a suitable subspace inside the GNS space of the centralizer which possesses `sufficient mixing' $($the associated subspace in \cite{bik-kunal, ER} was the GNS space of a MASA$)$, and use an appropriate orthonormal basis of that subspace to track a vector $\xi$, so that the cyclic subspace $\mathcal{A}\xi$ fully captures the `size' of the relative commutant of the centralizer. However, there is a payoff. Since we are working with a  `mixing subspace', we lose algebraic techniques (most importantly cannot locate unitaries) and depend on norm estimates of operators involving creation and annihilation operators. This forces us a bargain with the characteristic parameter that defines the two-dimensional representation, but our results remain valid for all $q\neq 0$. 

When $\mathcal{K}_\R\neq 0$, a similar idea with slight modification works. This time, the increment in dimension allows us to choose unitaries from an orthogonal subalgebra and frees us from any bargain with the characteristic parameter as was in the previous case, and the conclusions are no longer subject to any constraints. 

The key to the factoriality is Lemma \ref{sxi} in which `some mixing' is analyzed to control the relative commutant of the centralizer. The main results of this paper are summarized as follows:  \\ 

\noindent \textbf{Theorem}: Let $\CH_\R= \mathbb{R}^2\oplus \mathcal{K}_\R$, where $\mathcal{K}_\R$ is a real Hilbert space $($could be $0)$, and let $\mathbb{R}^2$ be reducing subspace for $(U_t)$ with associated sub representation being ergodic. Let $\lambda\in (0,1)$ be the parameter that defines $(U_t)$ on $\R^2$. Then, the following holds: 
\begin{enumerate}
\item if $dim(\CH_\R)=2$ and $\lambda$ is small (depending on $q$), then the centralizer of the $q$-quasi free $($vacuum$)$ state is irreducible if $q\neq 0$;
\item if $dim(\CH_\R)\geq 3$, the $q$-Araki-Woods algebras are factors for all $-1<q<1$;
\item if the almost periodic part of $(U_t)$ is sufficiently large, then the centralizer of the $q$-quasi free state is irreducible for all $-1<q<1$. \\
\end{enumerate} 
 
This paper is organized as follows. In \S\ref{Araki-Woods}, we lay out all the technical prerequisites that are needed to address the problem. The technical lemmas that will be used to deal with factoriality is divided into two groups under \S\ref{mainsec}; the case when $dim(\CH_\R)=2$ and $(U_t)$ is ergodic appears in \S\ref{casebig2} and that for all cases $(dim(\CH_\R)\geq 2)$ appears in \S\ref{casebig3}. The main theorems on factoriality appear in \S\ref{Secfactor}. Irreducibility of the centralizer is discussed in \S\ref{cent_irreducibel}.

\section{Preliminaries}\label{Araki-Woods}

In this section, we accumulate some well known facts about $q$-deformed Araki-Woods von Neumann algebras constructed by Hiai in \cite{Hiai} that will be indispensable for our purpose. As a convention, all Hilbert spaces in this paper are separable, all von Neumann algebras have separable preduals, inclusions of von Neumann algebras are unital and inner products are linear in the \emph{ second variable}. This section has overlap with \cite[\S2]{bik-kunal}.

\subsection{Hiai's Construction}\label{ConsHiai} Let $\CH_{\R}$ be a real Hilbert space with $\dim(\CH_{\R})\geq 2$ and let $(U_t)_{t \in \R}$ be a strongly continuous orthogonal representation of $\R$ on
$\CH_{\R}$.  Let $\CH_\C=\CH_{\R}\otimes_\R \C$ denote the
complexification of $\CH_{\R}$.  Denote the inner product and norm on
$\CH_\C $ by $\langle \cdot , \cdot \rangle_{\CH_\C} $ and
$\norm{\cdot}_{\CH_\C}$ respectively. Identify $\CH_{\R}$ in
$\CH_{\C}$ by $\CH_{\R} \otimes 1 $.  Thus, $\CH_\C = \CH_{\R} + i
\CH_{\R}$, and as a real Hilbert space the inner product of $\CH_{\R}
$ in $\CH_\C $ is given by $\Re \langle \cdot, \cdot \rangle_{\CH_\C}
$.  Consider the bounded anti-linear $($complex $($left$)$ conjugation$)$ operator $\mathcal{J}: \CH_\C
\rightarrow \CH_\C $ given by $\mathcal{J}(\xi + i \eta )= \xi - i
\eta$, $\xi, \eta \in \CH_{\R} $, and note that $\mathcal{J}\xi = \xi$
for $\xi\in \CH_{\R}$. Moreover,
\begin{align*}
\langle \xi, \eta \rangle_{\CH_\C} = \overline {\langle \eta , \xi \rangle}_{\CH_\C}  = \langle \eta , \mathcal{J} \xi \rangle_{\CH_\C}, \text{ for all } \xi \in \CH_{\C} , \eta \in \CH_{\R}.
\end{align*}
By abuse of notation we denote the linear extension of $(U_t)$ on $\CH_{\C}$ by the same notation, which is again a strongly continuous one-parameter group of unitaries in $\CH_{\C}$. Let $A$ denote the analytic generator of $(U_t)$. Then $A$ is positive,  nonsingular and self-adjoint. Note that $\CJ A = A^{-1}\CJ$. 

Introduce a new inner product on $\CH_{\C}$ by $\langle \xi, \eta \rangle_U = \langle \frac{2}{1+ A^{-1}} \xi, \eta \rangle_{\CH_\C}$, $\xi, \eta \in \CH_{\C}$, and let $\norm{\cdot}_{U}$ denote the associated norm on $\CH_{\C}$. Let $\CH$ denote the complex Hilbert space obtained by completing $(\CH_{\C}, \norm{\cdot}_{U})$.  The inner product and norm of $\CH$ will respectively be denoted by $\langle\cdot,\cdot\rangle_U$ and $\norm{\cdot}_{U}$ as well. Then, $(\CH_{\R}, \norm{\cdot}_{\CH_\C})\ni \xi \overset{\imath}\mapsto\xi \in (\CH_{\C}, \norm{\cdot}_{U})\subseteq (\CH,\norm{\cdot}_{U})$, is an isometric embedding of the real Hilbert space $\CH_\R$ in $\CH$ $($in the sense of \cite{Shlyakhtenko}$)$. With abuse of notation, we will identify $\CH_\R$ with its image $\iota(\CH_\R)$. Then, $\CH_{\R}\cap i\CH_{\R}=\{0\}$ and $\CH_{\R}+ i\CH_{\R}$ is dense in $\CH$ $($see pp. 332 \cite{Shlyakhtenko}$)$. 

Note that $(U_t)$ extends to a strongly continuous unitary representation $(\widetilde{U}_t)$ of $\R$ on $\CH$. Let $\widetilde{A}$ be the analytic generator associated to $(\widetilde{U}_t)$, which is obviously an extension of $A$. Any eigenvector of $\widetilde{A}$ is an eigenvector of $A$ corresponding to the same eigenvalue \cite[Prop. 2.1]{bik-kunal}. 
Since the spectral data of $A$ and $\widetilde{A}$ $($and hence of $(U_t)$ and $(\widetilde{U}_t))$ are essentially the same $($see \cite[\S2]{bik-kunal} for details$)$, and $\widetilde{U}_t,\widetilde{A}$ are respectively extensions of $U_t,A$ for all $t\in \R$, so we would now write $\widetilde{A}=A$ and $\widetilde{U}_t=U_t$ for all $t\in\R$. 

For $ q \in (-1, 1)$ and for the Hilbert space $\CH$, consider the associated 
$q$-Fock space $\CF_q(\CH)$ introduced in \cite{BS}. $\CF_q(\CH)$ is constructed as follows. Let $\Omega$ be a distinguished unit vector in $\C$ usually referred to as the vacuum vector. Denote $\CH^{\otimes 0}= \C\Omega$, and, for $n\geq 1$, let $\CH^{\otimes n} = \text{ span}_{\C}\{\xi_1 \otimes \cdots \otimes \xi_n : \xi_i \in \CH \text{ for }1\leq i\leq n\}$ denote the algebraic tensor products. Let $\CF_{fin}(\CH)=\text{ span}_{\C}\{\CH^{\otimes n}: n\geq 0 \}$. For $n,m\geq 0$ and $f=\xi_1 \otimes \cdots \otimes \xi_n \in \CH^{\otimes n}$, $g= \zeta_1\otimes \cdots \otimes \zeta_m\in \CH^{\otimes m}$, the association 
\begin{align}\label{qFock}
\langle f, g \rangle_q = \delta_{m, n} 
\sum_{ \pi \in S_n } q^{ i(\pi) } \langle \xi_1, \zeta_{\pi(1) }\rangle_U \cdots\langle \xi_n, \zeta_{\pi(n) } \rangle_U, 
\end{align} 
where $i(\pi)$ denotes the number of  inversions of the permutation $\pi \in S_n $, defines a positive definite sesquilinear form on  $ \CF_{fin}(\CH)$ and the $q$-Fock space $\CF_q(\CH) $ is the completion of $ \CF_{fin}(\CH)$ with respect to the norm $\norm{\cdot}_{q}$ induced by $ \langle \cdot, \cdot \rangle_q $. Denote $\CH^{\otimes_q n}:=
\overline{\CH^{\otimes n}}^{\norm{\cdot}_q}$, $n\geq 0$. Note that $\norm{\cdot}_{q}=\norm{\cdot}_{U}$ on $\CH^{\otimes_q 1}=\CH$. 

For $\xi\in \CH $, the left $q$-creation and $q$-annihilation operators on $\CF_q(\CH) $ are respectively defined by:
\begin{align}\label{Leftmult}
&c_q(\xi)\Omega  = \xi, \\ 
\nonumber&c_q(\xi) (\xi_1 \otimes \cdots\otimes \xi_n) =\xi \otimes  \xi_1 \otimes \cdots \otimes \xi_n,\\
\nonumber&\text{and},\\
\nonumber&c_q(\xi)^*\Omega  = 0,\\ 
\nonumber&c_q(\xi)^* (\xi_1 \otimes \cdots\otimes \xi_n) = \sum_{i = 1}^n{q^{i-1}}
\langle   \xi , \xi_i\rangle_U \xi_1 \otimes \cdots \otimes \xi_{i-1} \otimes \xi_{i+1} \otimes \cdots \otimes \xi_n, 
\end{align}
where $\xi_1 \otimes \cdots\otimes \xi_n\in \CH^{\otimes_{q} n}$ for $n\geq 1$.
The operators $c_q(\xi)$ and $c_q(\xi)^* $ are bounded on $ \CF_q(\CH) $ and they are 
adjoints of each other. Moreover,  
\begin{equation}\label{NormCreation}
\norm{c_q(\xi)} =\begin{cases} \frac{1}{\sqrt{1-q}} \norm{\xi}_U, \quad &\text{ if } 0\leq q < 1;\\
\norm{\xi}_U, \quad &\text{ if } -1 <  q \leq  0.\end{cases}	
\end{equation}
Moreover, they satisfy the following $q$-commutation relation:
\begin{align}\label{Commutationrelation}
c_q(\xi)^* c_q(\zeta) -q c_q(\zeta) c_q(\xi)^* = \langle \xi, \zeta  \rangle_U 1, \text{ for all } \xi,\zeta\in \CH.
\end{align}

The following Lemma from \cite{bik-kunal} will be crucial for our purpose. 

\begin{Lemma}\cite[Lemma 2.3]{bik-kunal}\label{SplitAdjoint}
Let $\xi, \xi_i, \eta_j \in \CH$, for $1\leq i\leq n$, $1\leq j\leq m$. Then, 
\begin{align*} 
c_q&(\xi)^* \bigg( (\xi_1  \otimes \cdots \otimes \xi_n) \otimes (\eta_1 \otimes \cdots \otimes \eta_m)\bigg)\\
&= \bigg(c_q(\xi)^*(\xi_1\otimes \cdots \otimes \xi_n)\bigg) \otimes (\eta_1 \otimes \cdots \otimes \eta_m)\\
&\indent\indent\indent\indent\indent+q^n 
(\xi_1  \otimes \cdots \otimes \xi_n) \otimes \bigg(c_q(\xi)^*(\eta_1 \otimes \cdots \otimes \eta_m) \bigg).
\end{align*} 
\end{Lemma}

Following \cite{Hiai, Shlyakhtenko}, consider the $C^*$-algebra 
$\Gamma_q( \CH_\R, U_t) := C^{*}\{ s_q(\xi) : \xi \in \CH_\R \}$ and the von Neumann algebra
$\Gamma_q( \CH_\R, U_t)^{\prime\prime}$, where 
\begin{align*}
s_q(\xi)  = c_q(\xi) + c_q(\xi)^* , \text{ }\xi \in \CH_\R.
\end{align*}  
$\Gamma_q( \CH_\R, U_t)^{\prime\prime}$ is known as the $q$-\textit{deformed Araki-Woods von Neumann algebra}  (see \cite[\S3]{Hiai}).  The vacuum state $\varphi_{q, U}:= 
\langle \Omega, \cdot\text{ } \Omega\rangle_q$ $($also called the $q$-\textit{quasi free state}$)$, is a faithful normal state of $\Gamma_q( \CH_\R, U_t)^{\prime\prime}$ and $\CF_{q}(\CH)$ is the GNS Hilbert space of $\Gamma_q( \CH_\R, U_t)^{\prime\prime}$ associated to $\varphi_{q,U}$. Thus, $\Gamma_q( \CH_\R, U_t)^{\prime\prime}$ acting on $\CF_{q}(\CH)$ is in standard form \cite{Hag75}.  

\textit{Making slight violation of the traditional notations, we will use the symbols $\langle \cdot,\cdot\rangle_q$ and $\norm{\cdot}_{q}$ respectively to denote the inner product and two-norm of elements of the GNS Hilbert space}.

\subsection{Modular Theory}\label{modular}
Most of what follows in \S\ref{modular} and \S\ref{Commute} is taken from \cite{Shlyakhtenko,Hiai}. We need to have a convenient description of the commutant and centralizer of $\Gamma_q( \CH_\R, U_t)^{\prime\prime}$ $($which has been recorded in the case $q=0$ in \cite{Shlyakhtenko} and a similar collection of operators in the commutant has been identified in \cite{Hiai}$)$. Thus, we need to record some facts related to the modular theory of the $q$-quasi free state $\varphi_{q,U}$.  Let $J_{\varphi_{q,U}}$ and $\Delta_{\varphi_{q,U}}$ respectively denote the modular conjugation and modular operator associated to $\varphi_{q,U}$ and let $S_{\varphi_{q,U}}=J_{\varphi_{q,U}}\Delta_{\varphi_{q,U}}^{\frac{1}{2}}$. Then, for $n\in\N$, 
\begin{align}\label{modulartheory}
&J_{\varphi_{q,U}}(\xi_1 \otimes \cdots \otimes \xi_n) = A^{-1/2}\xi_n \otimes \cdots \otimes A^{-1/2}\xi_1, \text{ }\forall\text{ } \xi_{i}\in \mathcal{H}_{\mathbb{R}}\cap \mathfrak{D}(A^{-\half});\\
\nonumber&\Delta_{\varphi_{q,U}}(\xi_1 \otimes\cdots \otimes \xi_n ) =
A^{-1}\xi_1 \otimes \cdots\otimes A^{-1}\xi_n, \text{ }\forall\text{ } \xi_{i}\in \mathcal{H}_{\mathbb{R}}\cap \mathfrak{D}(A^{-1});\\
	\nonumber&S_{\varphi_{q,U}}(\xi_1 \otimes \cdots \otimes \xi_n) =\xi_n \otimes \cdots \otimes \xi_1, \text{ }\forall\text{ } \xi_{i}\in \mathcal{H}_{\mathbb{R}}.
\end{align}
The modular automorphism group $(\sigma_{t}^{\varphi_{q,U}})$ of $\varphi_{q,U}$ is given by $\sigma_{-t}^{\varphi_{q,U}}=\text{Ad}(\CF(U_{t}))$, where $\CF(U_{t})=id\oplus \oplus_{n\geq 1} U_{t}^{\otimes_q n}$, for all $t\in \mathbb{R}$. In particular,
\begin{align}\label{modularaut}
\sigma^{\varphi_{q,U}}_{-t}(s_q(\xi)) = s_q(U_{t}\xi), \text{ for all } \xi \in \mathcal{H}_{\mathbb{R}}. 
\end{align}
To reduce notation, the complex $($left$)$ conjugation $\mathcal{J}$ associated to $\CH_\R\subset \CH$ will be denoted 
by $\overline{\xi+ i \eta}= \xi-i \eta$ for $\xi,\eta\in\CH_\R$. It corresponds with $S_{\varphi_q,U}$.

\subsection{Commutant}\label{Commute} Now we proceed to describe the commutant of $\Gamma_q( \CH_\R, U_t)^{\prime\prime}$. Consider the  set
\begin{align*}
\CH_\R' = \{ \xi \in \CH : \langle \xi, \eta \rangle_U \in \R \text{ for all }  \eta \in \CH_\R \}. 
\end{align*}
Then $\CH_\R^\prime$ is a real  subspace. Note that $\overline{ \CH_\R^\prime + i\CH_\R^\prime } = \CH$ and $ \CH_\R^\prime \cap i \CH^{\prime}_\R = \{ 0\}$. It is easy to check that $A^{-1/2} \zeta \in \CH_\R^\prime$  for all $\zeta \in \mathfrak{D}(A^{-\half})\cap \CH_\R$.

%%%%%%%% PROBABLY NOT NEEDED OR USED %%%%%%%%%%%%%%%%%%%%
%Let $ \zeta \in \mathfrak{D}(A^{-1/2})\cap \CH_\R$. Note that for all $\eta \in \CH_\R$, one has
%\begin{align}\label{A-Inner}
%\langle A^{-1/2} \zeta, \eta \rangle_U &=\langle \frac{2A^{-1/2}}{1+A^{-1}} \zeta, \eta \rangle_{\CH_\C}
%=\langle \eta,  \CJ \frac{ 2A^{-1/2}}{1+A^{-1}} \zeta\rangle_{\CH_\C}\\
%\nonumber&= \langle \eta,  \frac{ 2A^{1/2}}{1+A} \zeta \rangle_{\CH_\C}
%= \langle \frac{ 2}{1+A^{-1}}\eta,  A^{-1/2} \zeta \rangle_{\CH_\C}\\
%\nonumber& =\langle \eta,  A^{-1/2} \zeta \rangle_U.   
%\end{align}

%Also note that for $\eta,\xi\in \mathfrak{D}(A^{-1})\cap\CH_\R$, one has
%\begin{align}\label{A-Inner-1}
%\langle\eta,\xi \rangle_U &=\langle \frac{2}{1+A^{-1}} \eta, \xi \rangle_{\CH_\C}
%=\langle \xi,  \CJ \frac{ 2}{1+A^{-1}} \eta\rangle_{\CH_\C}\\
%\nonumber&= \langle \xi,  \frac{ 2}{1+A} \eta \rangle_{\CH_\C}=\langle \frac{ 2}{1+A^{-1}}\xi,  A^{-1} \eta\rangle_{\CH_\C}\\
%\nonumber&=\langle \xi,  A^{-1} \eta \rangle_U=\langle A^{-\half}\xi,  A^{-\half} \eta \rangle_U \indent\text{ (as }\mathfrak{D}(A^{-1})\subseteq \mathfrak{D}(A^{-\half})).
%\end{align}

Now for $\xi\in \CH$, define the right creation operator $c_{q, r}(\xi)$ on $\CF_q(\CH) $ by  
\begin{align}\label{Rightmult}
& c_{q, r}(\xi) \Omega = \xi, \\
\nonumber&c_{q, r}(\xi) (\xi_1 \otimes \cdots \otimes \xi_n) =  \xi_1 \otimes \cdots \otimes  \xi_n \otimes \xi, \text{ }\xi_{i}\in \CH, n\geq 1. 
\end{align}
Clearly, $c_{q, r}(\xi) = \jmath c_q(\xi) {\jmath}^{*}$, where $\jmath: \mathcal{F}_{q}(\CH) \rightarrow \mathcal{F}_{q}(\CH) $ is the unitary defined by 
\begin{align}\label{flip}
&\jmath ( \xi_1 \otimes \cdots \otimes \xi_n ) =  \xi_n \otimes \cdots \otimes \xi_1, \text{ where }\xi_{i}\in \CH \text{ for all }1\leq i\leq n, n\geq 1,\\
\nonumber&\jmath(\Omega)=\Omega.
\end{align}
Therefore, $c_{q, r}(\xi)\in \mathbf{B}(\CF_q(\CH)) $ and its adjoint $c_{q, r}(\xi)^*$ is given by 
\begin{align}\label{Rightmultadj}
&c_{q, r}(\xi)^* \Omega = 0, \\
\nonumber&c_{q, r}(\xi)^* ( \xi_1 \otimes \cdots \otimes \xi_n) = \sum_{i = 1}^n q^{n-i}\langle  \xi , \xi_i \rangle_U \xi_1 \otimes \cdots \otimes \xi_{i-1} \otimes \xi_{i+1} \otimes \cdots \otimes \xi_n, \text{ }\xi_{i}\in \CH, n\geq 1. 
\end{align} 

Write $ s_{q, r}(\xi) = c_{q, r}(\xi) + c_{q, r}(\xi)^*$, $\xi\in\CH$. The following result describes the commutant of $\Gamma_q(\CH_\R, U_t)^{\prime\prime}$.

\begin{Theorem}\cite[Thm. 2.4]{bik-kunal}\label{commutant}
Suppose $ \xi \in \dom(A^{-1})\cap \CH_\R $. Then $J_{\varphi_{q,U}}s_q(\xi) J_{\varphi_{q,U}} = s_{q, r}(A^{-\fhalf} \xi)$. Moreover, $\Gamma(\CH_{\R}, U_{t})^{\prime}=\{s_{q, r}(\xi): \xi\in \CH_\R^\prime\}^{\prime\prime}$.
\end{Theorem}
The complex $($right$)$ conjugation associated to $\CH_\R'\subset \CH$ will be denoted by $\overline{\xi+ i \eta}^{\,r}=\xi-i \eta$ for $\xi,\eta\in\CH_\R'$. It corresponds with $\jmath S_{\varphi_q,U}\jmath^*$.  

\subsection{Notations and some technical facts}\label{not}

In this paper, we are interested in the factoriality of $\Gamma_q(\CH_\R, U_t)^{\prime\prime}$ and the orthogonal representation remains arbitrary but fixed. Thus, to reduce notation, we will write $M_q=\Gamma_q(\CH_\R, U_t)^{\prime\prime}$ and $\varphi=\varphi_{q,U}$. We will also denote $J_{\varphi_{q,U}}$ by $J$ and $\Delta_{\varphi_{q,U}}$ by $\Delta$. As $\Omega$ is separating for both $M_q$ and $M_q^{\prime}$, for $\zeta\in M_q\Omega$ and $\eta\in M_q^{\prime}\Omega$ there exist unique $x_{\zeta}\in M_q$ and $x^{\prime}_{\eta}\in M_q^{\prime}$ such that $\zeta =x_{\zeta}\Omega$ and $\eta=x_{\eta}^{\prime}\Omega$. In this case, we will write 
\begin{align*}
W(\zeta)=x_{\zeta} \text{ and } W_r(\eta)=x_{\eta}^{\prime}. 
\end{align*}
Note that $W_r(\eta)= JW(J\eta)J$, as $J\eta\in M_q\Omega$ from Tomita's fundamental theorem. Thus, for example, as $\xi\in M_q\Omega$ for every $\xi\in \CH_\R$, so $W(\xi+i\eta)=s_q(\xi)+is_q(\eta)$ for all $\xi,\eta\in\CH_\R$.

Write $\mathcal{Z}(M_q)= M_q\cap M_q^{\prime}$. Let $M_q^{\varphi}=\{x\in M_q: \sigma_t^{\varphi}(x)=x \text{ for all }t\in\R\}$ denote the centralizer of $M_q$ associated to the state $\varphi$. Recall that $x\in M_q$ is analytic with respect to $(\sigma_t^{\varphi})$ if and only if the function 
$\R \ni t\mapsto \sigma_t^{\varphi}(x)\in M_q$ extends to a weakly entire function. We say that a vector $\xi\in M_q\Omega$ is \textit{analytic}, if $W(\xi)$ is analytic for $(\sigma_t^{\varphi})$. 

In order to control calculations within the page limit, we adopt the following notations for convenience.
\begin{enumerate}
\item $\xi_1\cdots\xi_n := \xi_1\otimes\cdots\otimes \xi_n \in \CH^{\otimes_q n}$ for 	 $ \xi_i\in \CH $, $1\leq i\leq n$;
\item $c_q(\xi)=c_q(\xi_1)\cdots c_q(\xi_n)$ for $\xi=\xi_1\cdots\xi_n\in \CH^{\otimes_q n}$;
\item $c_{q,r}(\xi)=c_{q,r}(\xi_n)\cdots c_{q,r}q(\xi_1)$ for
  $\xi=\xi_1\cdots\xi_n\in \CH^{\otimes_q n}$;
\item $C_q= \prod_{i=1}^\infty \frac 1{1-|q|^i}$;\\
\item $d_0=1$, $d_j=\prod_{i=1}^j (1-q^i)$, $j\in\N$, and $d_\infty =\prod_{i=1}^\infty (1-q^i)$; \\	 
\item $[n]_q := 1+ q+ \cdots +  q^{(n-1)}$, $[n]_q! :=\prod_{j=1}^{n} [j]_q, \text{ for }n\geq 1$,    
and $[0]_q := 0$, $[0]_q! := 1$ by convention.
\end{enumerate}
{Note that $d_k,\frac{1}{d_k}\leq C_q$ for all $k\geq 0$ and $\abs{q}<1$}.

The following norm estimates will be crucial (see  \cite{BKS,BS,B99,ER}).\\
\noindent$\bullet$ For all  $\xi \in \CH^{\otimes_q n }$, 
\begin{align}\label{Normelt}
\| c_q(\xi)\|\leq \sqrt{C_q} \|\xi\|_q.
\end{align}
\noindent$\bullet$ If $\xi\in \CH$ and ${\norm{\xi}}_U = 1(=\norm{\xi}_q)$, then 
\begin{align}\label{Normelt2}
{\norm{\xi^{ n}}}_q^2 = [n]_q! = d_n(1-q)^{-n}. 
\end{align}
\noindent$\bullet$ If $\xi_1,\cdots, \xi_n , \xi \in \CH$ with $ {\norm{\xi_j}}_U={\norm{\xi}}_U = 1$ for all $1\leq j\leq n$, then 
\begin{align}\label{normestimate}
{\norm{\xi_1 \cdots \xi_n \xi^{ m}}}_q  = {\norm{ \xi^{ m} \xi_n \cdots \xi_1}}_q
\leq C_q^{\frac{n}{2}} \sqrt{ [m]_q!}, \text{ }m\geq 0.
\end{align}

\noindent$\bullet$
We recall the following  $q$-analogue of the Pascal's identity for $q$-binomial coefficients (cf. \cite[Prop. 1.8]{BKS}):
\begin{align}\label{Pascal}
q^k \left( \begin{matrix} n\\ k \end{matrix}\right)_q + \left( \begin{matrix} n\\ k-1 \end{matrix}\right)_q = \left( \begin{matrix} n+1 \\ k \end{matrix}\right)_q,\quad  k \leq n.
\end{align}
We also recall the Wick formula from \cite{B99, nou}, and its right version which can be obtained using the $($right$)$ complex  conjugation.

\begin{Proposition}\label{wick}
Let $\xi_1,\cdots,\xi_n$ be in $\CH_\C$. Then,
$$W(\xi_1\cdots \xi_n)=\sum_{i=0}^n \sum_{\sigma\in S_{n,i}} q^{\abs{\sigma}}c_q(\xi_{\sigma(1)})\cdots
c_q(\xi_{\sigma(i)})c_q(\overline{\xi_{\sigma(i+1)}})^*\cdots
c_q(\overline{\xi_{\sigma(n)}})^*,$$
where $S_{n,i}$ is the set of permutations of $\{1,\ldots,n\}$ that are increasing on $\{1,\ldots,i\}$ and $\{i+1,\ldots,n\}$ and $|\sigma|$ is the number of inversions of $\sigma$. Further, if $\xi_1,\cdots,\xi_n\in\CH_\R^\prime + i\CH_\R^\prime$, then 
$$W_r(\xi_1\cdots \xi_n)=\sum_{i=0}^n \sum_{\sigma\in S_{n,i}} q^{\abs{{\rm flip\,}\circ\,\sigma}}c_{q,r}(\xi_{\sigma(1)})\cdots
c_{q,r}(\xi_{\sigma(i)})c_{q,r}(\overline{\xi_{\sigma(i+1)}}^{\,r})^*\cdots c_{q,r}(\overline{\xi_{\sigma(n)}}^{\,r})^*,$$
where ${\rm flip}$ is the permutation ${\rm flip}(k)=n-k$, $1\leq k\leq n$.
\end{Proposition}

There is also another convenient way to write the Wick formula using crossings of partitions.

Any $\sigma\in S_{n,i}$ is completely determined by a subset $\mathfrak{J}=\{j_1<\cdots<j_i\}$ with complement $\mathfrak{J}^c=\{k_{i+1}<\cdots<k_n\}$. Then $\abs{\sigma}$ is the number of crossings of the partition $\mathfrak{J}\cup \mathfrak{J}^c$, i.e., 
$$\abs{\sigma}=c(\mathfrak{J},\mathfrak{J}^c)=\#\{(a,b) \;|\; j_a>k_b\}.$$ 
Thus,
\begin{equation}\label{wick'}
W(\xi_1\cdots \xi_n)=\sum_{i=0}^n 
\sum_{\substack{\mathfrak{J}=\{j_1<\cdots<j_i\}\\ \mathfrak{J}^c=\{k_{i+1}<\cdots<k_n\}\\ \mathfrak{J}\cup \mathfrak{J}^c=\{1,...,n\}}}
 q^{c(\mathfrak{J},\mathfrak{J}^c)}c_q(\xi_{j_1})\cdots c_q(\xi_{j_i})c_q(\overline{\xi_{k_{i+1}}})^*\cdots
c_q(\overline{\xi_{k_n}})^*.
\end{equation}

%%%%%%%%%%%%%%%%%%%%%%%%%%%%%%%%%%%%%%%%%%%%%%%%%%%%%%%%%%%%%%%%%%%%%%%%%%%%%%%%%%%%%%%%%%%%%%%%%%%%%%%
%%%%%%%%%%%%%%%%%%%%%%%%%%%%%%%%%%%%%%%%%%%%%%%%%%%%%%%%%%%%%%%%%%%%%%%%%%%%%%%%%%%%%%%%%%%%%%%%%%%%%%%

\subsection{Centralizer}\label{SectionCentralizer}

We need a convenient description of the centralizer $M_q^\varphi$ which depends on the almost periodic part of the orthogonal representation $(U_t)$. We need some preparation. For details check \cite{bik-kunal}. 

Recall that for a  strongly continuous orthogonal representation $t\mapsto U_t$, $t\in \R$,  on the real Hilbert space $\CH_\R$, there is a unique decomposition $($see \cite{Shlyakhtenko}$)$,
\begin{align*}%\label{Representation}
( \CH_\R, U_t)  =  \left(\bigoplus_{j = 1}^{N_1} (\R, \text{id} )\right) \oplus \left( \bigoplus_{k=1}^{N_2}(\CH_\R(k), U_t(k) ) \right) \oplus (\widetilde{\CH}_\R, \widetilde{U}_t),
\end{align*}
where $0\leq N_1,N_2\leq \aleph_0$, 
\begin{align*}
\CH_\R(k) = \R^2, \quad U_t(k)  = \left( \begin{matrix} \cos( t\log \lambda_k)& - \sin( t\log\lambda_k) \\
\sin( t\log\lambda_k)& \cos( t\log\lambda_k)  \end{matrix} \right), \text{ }0<\lambda_k < 1,
\end{align*}
and $(\widetilde{\CH}_\R, \widetilde{U}_t)$ corresponds to the weakly mixing component of the orthogonal representation; thus $\widetilde{\CH}_\R$ is either $0$ or infinite-dimensional.

If $N_1\neq 0$, let $e_j=0\oplus\cdots\oplus 0\oplus 1\oplus 0\oplus\cdots\oplus 0\in \bigoplus_{j=1}^{N_1}\R$, where $1$ appears at the $j$-th place for $1\leq j\leq N_1$. Similarly, if $N_2\neq 0$, let $f_k^1=0\oplus\cdots\oplus 0\oplus\left( \begin{matrix}1\\0 \end{matrix}\right)\oplus 0\oplus\cdots \oplus 0\in\bigoplus_{k=1}^{N_2}\CH_\R(k)$ and $f_k^2 =0\oplus\cdots\oplus 0\oplus\left( \begin{matrix}0\\1\end{matrix}\right)\oplus 0\oplus\cdots \oplus 0\in \bigoplus_{k=1}^{N_2}\CH_\R(k)$ be vectors with nonzero entries in the $k$-th position for $1\leq k\leq N_2$. Denote 
\begin{align*}
e^1_k= \frac{\sqrt{\lambda_k +1}}{2}(f_k^1+if_k^2) \text{ and }e^2_k=\frac{\sqrt{{\lambda}^{-1}_k+1}}{2}(f_k^1-if_k^2);
\end{align*}
thus $e_k^1,e_k^2 \in\CH_\R(k)+i\CH_\R(k)$ are orthonormal basis of $(\CH_\R(k)+i\CH_\R(k),\langle\cdot,\cdot\rangle_{U})$ for $1\leq k\leq N_2$. Fix $1\leq k\leq N_2$. The analytic generator $A(k)$ of $(U_t(k))$ is given by 
\begin{align*}
A(k)=\frac{1}{2}\left( \begin{matrix} \lambda_k + \frac{1}{\lambda_k} & i(\lambda_k - \frac{1}{\lambda_k})\\
		-i(\lambda_k -\frac{1}{ \lambda_k}) &\lambda_k + \frac{1}{\lambda_k}
\end{matrix}\right).
\end{align*}
Moreover, 
\begin{align*}
A(k) e_k^1 = \frac{1}{\lambda_k} e_k^1  \text{ and }  A(k) e_k^2 = \lambda_k e_k^2.
\end{align*}
Write $\CS=\{e_j: 1\leq j\leq N_1\} \cup \{e_k^1,e_k^2: 1\leq k\leq N_2\}$ if $N_1\neq 0$ or $N_2\neq 0$, else set $\CS=\{0\}$. If $\CS\neq \{0\}$, then $\CS$ is an orthogonal set in $(\CH_\C, \langle \cdot, \cdot\rangle_U)$ and the collection of eigenvectors of the analytic generator $A$ of $(U_t)$ is contained in $\text{span }\CS$. In the event $\CS\neq \{0\}$, 
rename the elements of the set $\CS$ as $\zeta_1,\zeta_2,\cdots$, i.e., $\CS=\{\zeta_i:1\leq i\leq N_1+2N_2\}$, whence $A\zeta_l=\beta_l\zeta_l$ with $\beta_l\in\mathcal{E}_A$ for all $l$, where $\mathcal{E}_A=\{1\}\cup \{\lambda_k:1\leq k\leq N_2\}\cup \{\frac{1}{\lambda_k}:1\leq k\leq N_2\}$. It is to be understood that when $N_1=\aleph_0$ $($resp. $N_2=\aleph_0)$, the constraints $j\leq N_1$ and $i\leq N_1+2N_2$ $($resp. $k\leq N_2$ and $i\leq N_1+2N_2)$ $($in defining $\CS$ and $\mathcal{E}_A)$ is replaced by $j<N_1$ and $i<N_1+2N_2$ $($resp.  $k<N_2$ and $i<N_1+2N_2)$. 

Now we are ready to write down the description of $M_q^\varphi$. 

\begin{Theorem}\cite[Thm. 3.4]{bik-kunal}\label{CentraliserDescribe}
Let 
\begin{align*}
\CW_0 =\begin{cases}& \{\zeta_{i_1}\cdots \zeta_{i_n}:\zeta_{i_j} \in \CS, 1\leq i_j\leq  N_1+2N_2, \prod_{j=1}^n\beta_{i_j}= 1, n\in \N \},\\
&\indent\indent\indent\indent\indent\indent\indent\indent\indent \text{ if }\max(N_1,N_2)<\infty;\\
&\{\zeta_{i_1}\cdots\zeta_{i_n}:\zeta_{i_j} \in \CS, 1\leq i_j <N_1+2N_2, \prod_{j=1}^n\beta_{i_j}= 1, n\in \N \},\\
&\indent\indent\indent\indent\indent\indent\indent\indent\indent \text{ if }\max(N_1,N_2)=\infty.\\
\end{cases}
\end{align*} 
Let $\CW=\C\Omega\oplus \overline{\text{span }\CW_0}^{\norm{\cdot}_q}$. Then, $M_q^\varphi\Omega=\CW\cap M_q\Omega$.
\end{Theorem}

\begin{Remark}\label{CentralizerRemark}
If $\mathcal{E}_A=\{1\}$, then $M_q^{\varphi}$ is isomorphic to the $q$-Gaussian von Neumann algebra of Bo$\overset{.}{\text{z}}$ejko and Speicher \cite{BS,BKS}. Thus, in this case, if $1$ is eigenvalue of multiplicity more than or equal to $2$, then $M_q^{\varphi}$ is a factor by \cite{ER}.
\end{Remark}

\section{Technical Analysis}\label{mainsec}

The factoriality of $M_q$ is decided in several steps. When $\dim(\CH_\R)=2$, then $(U_t)$ is not ergodic if and only if it is trivial and in that case $M_q^\varphi=M_q$ is a $\rm{II}_1$ factor $($see \cite{bik-kunal,ER}$)$. The most important and difficult case is the one when $\dim(\CH_\R)=2$ and $(U_t)$ is ergodic. The difficulty arises in lack of room to perform meaningful calculations. 

In this section, we lay out the technical analysis that will lead to the factoriality of $M_q$. This section is divided into two subsections. In \S\ref{casebig2} we deal with the case when $\dim(\CH_\R)=2$ and $(U_t)$ is ergodic and in \S\ref{casebig3} we prepare the machinery that will help deal with all the cases.

\subsection{$\mathbf{\Gamma(\R^2, U_t )^{\prime\prime}}$ with small $\lambda$ and $(U_t)$ ergodic}\label{casebig2}

Following the discussion in \S\ref{SectionCentralizer}, it follows that $N_1=0$ and $N_2=1$. Further, there exists a $\lambda\in (0,1)$ such that 
\begin{align*}
U_t  = \left( \begin{matrix} \cos( t\log \lambda)& - \sin( t\log\lambda) \\
\sin( t\log\lambda)& \cos( t\log\lambda)  \end{matrix} \right).
\end{align*}
Reducing notation, write $e_1:=e_1^1=\frac{\sqrt{\lambda+1}}{2}\begin{pmatrix}1\\i\end{pmatrix}$ and $e_2:= e_1^2=\frac{\sqrt{\lambda^{-1}+1}}{2}\begin{pmatrix}1\\-i\end{pmatrix}$ so that $Ae_1 = \frac{1}{\lambda} e_1  \text{ and }  A e_2 = \lambda e_2$ and $\{ e_1, e_2\}$ forms a orthonormal basis of $(\C^2, \langle \cdot, \cdot \rangle_U)$.

Now we set $ e = \lambda^{-\frac{1}{4}}e_1$. Then, $\bar{e}= \lambda^{\frac{1}{4}}e_2$. Note that   $\norm{e}_U = \lambda^{-\frac{1}{4}}$ and  $\norm{\bar{e}}_U = \lambda^{\frac{1}{4}}$. Note that in this context $ \CH = \C^2$ considered with respect to the inner product $ \langle \cdot, \cdot \rangle_U$. We also have the formulas  $\bar{e}^{\,r}=\lambda^{-1} \bar e$ and $\overline{\,\bar{e}\,}^{\,r}=\lambda e$.

%Some time we write  $ \langle \cdot, \cdot \rangle_U$ just as $ \langle \cdot, \cdot \rangle$. 

We begin with some useful lemmas. 
\begin{Lemma}\label{Adjpowers}
For $m\geq 0$ and $\xi\in \CH$, we have  
\begin{equation*}
c_q(\xi)^{*n}(\xi^m)=\begin{cases}\frac{[m]_q!}{[m-n]_q!} \norm{\xi}_q^{2n} \xi^{m-n}, \quad &n\leq m;\\
0, \quad  &n>m.\end{cases} 
\end{equation*} 
\end{Lemma}
\begin{proof}
The proof follows directly from Eq. \eqref{Leftmult}.
\end{proof}

\begin{Lemma}\label{wen}
$	W(e^n) = W(e)^n =\displaystyle \sum_{k=0}^{n}  \left( \begin{matrix} n\\ k \end{matrix}\right)_q c_q(e)^{n-k} c_q(\bar{e})^{* k}$ for $n\geq 1$.
\end{Lemma}

\begin{proof}
The proof is straightforward and follows by using Eq. \eqref{Pascal} and induction. So we omit the proof.  
\end{proof}

The next lemma initiates the interplay between the parameters $\lambda$ and $q$ in the case $\dim(\CH_\R)=2$, which will force constraints to the factoriality of $M_q$. 

\begin{Lemma}\label{bnd1}
Let $D(q)=\sup_n d_n$. With the above notations, the following are true:
\begin{enumerate}
\item[$(i)$] $\{\lambda^{\frac{n}{4}} (1-q)^{\frac{n}{2}} c_q(e)^n \}_{n\geq 1}$, $\{\lambda^{-\frac{n}{4}} (1-q)^{\frac{n}{2}} c_q(\bar e)^n\}_{n\geq 1} $ are bounded $($actually by $1$ if $q\geq 0$ and by $\sqrt{C_qD(q)}$ if $q<0)$.
\item[$(ii)$] $ \{\lambda^{\frac{n}{4}} (1-q)^{\frac{n}{2}} W(e^{ n} )\}_{n\geq 1} $ is bounded.
\end{enumerate}
\end{Lemma}	

\begin{proof}
\noindent $(i)$. First assume $0\leq q<1$. Then, from Eq. \eqref{NormCreation} it follows that 
\begin{align*}
\norm{c_q(e)^n} \leq \norm{c_q(e)}^n = \norm{e}_q^n (1-q)^{-\frac{n}{2}} = \lambda^{-\frac{n}{4}}(1-q)^{-\frac{n}{2}}.
\end{align*}
If $-1<q<0$, from Eq. \eqref{Normelt} and Eq. \eqref{Normelt2} it follows that 
\begin{align*}
\norm{c_q(e)^n}&=\norm{c_q(e^n)}\leq \sqrt{C_q} \norm{e^n}_q\\
& = \sqrt{C_q}\norm{e}_q^n \sqrt{[n]_q!}\\
&=\sqrt{C_q}\lambda^{-\frac{n}{4}} \sqrt{[n]_q!}\\
%&=\sqrt{C_q}\lambda^{-\frac{n}{4}} (1-q)^{-\frac{n}{2}} \Big(\prod_{k=1}^n (1-q^k)\Big)^{\frac{1}{2}}\\
&=\sqrt{C_qd_n}\lambda^{-\frac{n}{4}} (1-q)^{-\frac{n}{2}}\leq \sqrt{C_qD(q)}\lambda^{-\frac{n}{4}} (1-q)^{-\frac{n}{2}}. 
\end{align*}
The argument for $\bar e$ is similar. Thus $(i)$ follows.

\noindent $(ii)$. We just use  Lemma \ref{wen} and the triangle inequality. Note that 
\begin{eqnarray*}
\| W(e^n)\| &\leq &\sum_{k=0}^{n}  \frac{d_n}{d_{n-k}d_k}  \|c_q(e)^{n-k}\|. \|c_q(\bar{e})^{* k}\| \quad (\text{use Eq. \eqref{Normelt2}})\\
& \leq & C_q \sum_{k=0}^{n} \frac{d_n}{d_{n-k}d_k} \|e^{n-k}\|_q.\|\bar{e}^{k}\|_q \quad (\text{use Eq. \eqref{Normelt}})\\
& = &C_q \sum_{k=0}^n  \frac{d_n}{d_{n-k}d_k} \|e\|_q^{n-k} \sqrt{[n-k]_q!}
\|\bar e\|_q^{k} \sqrt{[k]_q!}\quad (\text{use Eq. \eqref{Normelt2}})\\
&= & C_q   \sum_{k=0}^n  \frac{d_n}{\sqrt{d_{n-k}d_k}} \lambda^{(2k-n)/4}
(1-q)^{-n/2} \quad (\text{use Eq. \eqref{Normelt2}}).
\end{eqnarray*}
The result follows as $(d_i)$ is bounded from above and below and $\lambda<1$.
\end{proof}

\begin{Lemma}\label{Tn}
Let $T_n=(1-q)^n\lambda^{n/2}c_q(e)^{* n } c_q(e)^n$ for $n\geq 1$, then $T_n$ is norm convergent to some $T\in C^*\langle c_q(e)\rangle$.

Moreover, if  $q\geq0$ then $ d_\infty \leq  T\leq 1$ and if $q\leq0$ then $d_\infty/(1-q) \leq T \leq 1-q$.
\end{Lemma}	

\begin{proof}
Let $\mathcal{K}=\C e$. Then by \cite[Thm. 3.6]{W17} (see also \cite{W18}, \cite{JSW} Example 2), the $C^*$-algebras generated by the
creation operators associated to $e$ in $\mathcal F_q(\mathcal
H)$ and in $\mathcal F_q(\mathcal K)$ are isomorphic. Thus, we only need to prove the result in $\mathbf
B(\mathcal F_q(\mathcal K))$; but $\mathcal F_q(\mathcal K)= \oplus_{k=0}^\infty
\C e^{k}$. With respect to this decomposition, $T_n$ is a
diagonal operator with 
\begin{align*}
T_n( e^{ k})&= (1-q)^n\lambda^{n/2} c_q(e)^{* n } (e^{(n+k)})
= (1-q)^n \frac{[n+k]_q!}{[k]_q!} e^{k}\\
&= \Big(\prod_{j=1}^n (1-q^{k+j})\Big)e^{k},\quad \forall k\geq 0.
\end{align*}
Let $T\in\mathbf{B}(\mathcal F_q(\mathcal K))$ be the diagonal operator with eigenvalues $\frac{d_\infty}{d_k}=\prod_{j=1}^\infty (1-q^{k+j})$ and associated eigenvectors $e^{k}$ for all $k\geq 0$. Then, 
\begin{align*}
\| T_n-T\|\leq K\sup_{k} |1- \prod_{j=n+1}^\infty  (1-q^{k+j})| \to 0\text{ as }n\to \infty, 
\end{align*}
where $K$ is a constant (independent of $n,k$). Consequently, $0\leq T\in C^*\langle c_q(e)\rangle$. 

The estimations for $T$ are clear from its spectrum if $q\geq 0$. When $q<0$, we have to estimate $\alpha_q=\inf_{k}\prod_{j=1}^\infty (1-q^{k+j})$ and $\beta_q=\sup_{k}\prod_{j=1}^\infty (1-q^{k+j})$. Note that 
for all $m\geq0$, we have that $(1-q^{2m})(1-q^{2m+1})\leq 1$. Thus, it follows that $\alpha_q=\prod_{j=2}^\infty (1-q^{j})=d_\infty/(1-q)$.  In the same way since $\prod_{j=2m}^\infty (1-q^{k+j})\leq 1$, we must have $\beta_q\leq 1-q$.
\end{proof}

\begin{Remark}\label{estofT}
When $q<0$, if moreover $|q| (1+|q|)\leq 1$, we have that for all $m\geq 0$, 
$(1-q^{2m+1})(1-q^{2m+2})\geq 1$. Thus, it follows that $\beta_q=\sup_{k} d_\infty/d_k=d_\infty$ in that case. 
\end{Remark}

Now, consider the following operator $$ S_n = (1-q)^n \lambda^{ \frac{n}{2} }  {c_q(e)}^{ *n} W(e^{n}).$$
Also note that, for each fixed $k$, as $n$ goes to infinity, $\left( \begin{matrix} n\\ k \end{matrix}\right)_q=\frac{d_{n}}{d_{n-k}d_k}$ converges to  $c_k=d_k^{-1}\leq C_q$.

\begin{Lemma}\label{s}
The sequence $S_n$ converges in norm to $$\displaystyle{S_\infty =\sum_{k=0}^\infty c_k (1-q)^k\lambda^{k/2} c_q(e)^{*k}Tc_q(\overline e)^{k*}}.$$ Moreover, if $q>0$ and $\lambda<(1+\frac{C_q}{d_\infty})^{-2}$ or if $q<0$ and $\lambda<(1+\frac{C_q^2D(q)(1-q)^2}{d_\infty})^{-2}$, the operator $S_\infty$ is invertible.
\end{Lemma}

\begin{proof}
Let $B(q)= \sqrt{C_qD(q)}$. First note that the series defining $S_\infty$ converges absolutely. Indeed by Lemma \ref{bnd1} $(i)$, we have 
\begin{align*}
&\| (1-q)^{k/2}c_q(e)^k \|\leq \lambda^{-k/4}, \quad \| (1-q)^{k/2}c_q(\overline e)^k \|\leq\lambda^{k/4}, \quad\text{ when } 0\leq q<1;\\
&\| (1-q)^{k/2}c_q(e)^k \|\leq B(q)\lambda^{-k/4}, \,\, \| (1-q)^{k/2}c_q(\overline e)^k \|\leq B(q) \lambda^{k/4}, \text{ when } -1< q<0;
\end{align*}
Thus, using Lemma \ref{Tn} we have
\begin{equation*}
\sum_{k=0}^\infty \| c_k (1-q)^k\lambda^{k/2} c_q(e)^{*k}Tc_q(\overline e)^{k*} \| \leq \begin{cases} C_q \sum_{k=0}^\infty {\lambda}^{k/2}<\infty, &\text{ if } q\geq 0;\\ C_q^2D(q)(1-q)\sum_{k=0}^\infty {\lambda}^{k/2}<\infty, &\text{ if } q< 0. \end{cases}
\end{equation*}
We note the following:
\begin{align*}
S_n &= (1-q)^n \lambda^{ \frac{n}{2} }  {c_q(e)}^{ *n} W(e^{ n})\\
&=(1-q)^n \lambda^{ \frac{n}{2} }  {c_q(e)}^{ *n} 	\displaystyle \sum_{k=0}^{n}  \left( \begin{matrix} n\\ k \end{matrix}\right)_q c_q(e)^{n-k} c_q(\bar{e})^{* k} \quad (\text{Lemma \ref{wen}})\\
&=	\displaystyle \sum_{k=0}^{n}  \left( \begin{matrix} n\\ k \end{matrix}\right)_q  (1-q)^k \lambda^{ \frac{k}{2} }  {c_q(e)}^{ *k} (1-q)^{n-k} \lambda^{\frac{n-k}{2}}  c_q(e)^{* (n-k)}  c_q(e)^{n-k} c_q(\bar{e})^{* k}\\
&=	\displaystyle \sum_{k=0}^{n}  \left( \begin{matrix} n\\ k \end{matrix}\right)_q  (1-q)^k \lambda^{ \frac{k}{2} }  {c_q(e)}^{ *k} T_{n-k} c_q(\bar{e})^{* k} \quad (\text{Lemma \ref{Tn}}).
\end{align*}
We have just seen in Lemma \ref{Tn} that $T_{n-k}$ is norm convergent to $T$ (and bounded by $1$ if $q\geq 0$ or  $1-q$ if $q<0$). Thus, the general term in $S_n$ converges to the general term of $S_\infty$ in norm and one concludes the convergence in the statement with the dominated convergence theorem as $\sum_{k=0}^\infty \lambda^{k/2}<\infty$.

We also have $S_\infty=T(1+T^{-1}V)$ $($use Lemma \ref{Tn}$)$, where 
\begin{equation*}
\|V\| \leq \begin{cases} C_q \sum_{k=1}^\infty {\lambda}^{k/2},\quad &\text{ if } q>0;\\ C_q^2D(q)(1-q) \sum_{k=1}^\infty {\lambda}^{k/2},\quad &\text{ if } q<0.\end{cases} 
\end{equation*}
The rest follows by ensuring that $\norm{T^{-1}V}<1$ using Lemma \ref{Tn}. We know that $\|T^{-1}\|\leq \frac{1}{d_\infty}$ if $q>0$ and $\|T^{-1}\|\leq \frac{1-q}{d_\infty}$ if $q<0$ by Lemma \ref{Tn}.  Thus, $S_\infty$ is invertible if $\frac{C_q}{d_\infty} \frac {\sqrt \lambda}{1-\sqrt{\lambda}} <1$ when $q>0$ and $C_q^2D(q)\frac{(1-q)^2}{d_\infty} \frac {\sqrt \lambda}{1-\sqrt{\lambda}} <1$ when $q<0$.
\end{proof}	

\begin{Remark}\label{Rem1} $ $
\begin{enumerate}
\item When $q=0$, $S_\infty=\sum_{k=0}^\infty \lambda^{k/2} c_0(e)^{*k}c_0(\bar e)^{*k}$. The operators $l_k=c_0(\bar e)^kc_0( e)^{k}$ are partial isometries with orthogonal ranges {when $k\geq 1$}. On the smallest subspace containg $\Omega$ and invariant by $l_k$'s,  $S_\infty$ acts like $ 1+ \sqrt{\frac \lambda {1-\lambda}}l^*$ where $l$ is a unilateral shift. Hence, $S_\infty$ has a kernel if  $\frac \lambda {1-\lambda}>1$, that is $\lambda>\frac 12$.
\item As a result of Lemma \ref{s}, we have 
\begin{align*}
S_\infty^*\Omega&= \sum_{k=0}^\infty c_k(1-q)^k\lambda^{k/2} c_q(\bar e)^kTc_q(e)^k\Omega
=\sum_{k=0}^\infty c_k(1-q)^k\lambda^{k/2} c_q(\bar e)^kT(e^k)\\
&=\sum_{k=0}^\infty c_k(1-q)^k\lambda^{k/2} c_q(\bar e)^k\frac{d_\infty}{d_k}e^k \quad(\text{Lemma \ref{Tn}})\\
&=d_\infty\sum_{k=0}^\infty c_k^2(1-q)^k\lambda^{k/2}{\bar e}^ke^k.
\end{align*}
The vector $S_\infty^*\Omega$ will be useful later in the proof. 
\end{enumerate}
\end{Remark}

As a consequence of the Wick formula, we explicitly get:
\begin{Lemma}\label{ween}
For every $n\geq0$, there are reals $q_{k,l}$, $0\leq k,l\leq n$ with $|q_{k,l}|\leq C_q^2 |q|^{(n-k)l}$ such that
$$W(\bar e^n e^n)=\sum_{k,l=0}^n q_{k,l}  c_q(\bar e)^kc_q( e)^l c_q(e)^{* {(n-k)}}c_q(\bar e)^{*(n-l)}.$$
%\textcolor{red}{? $c_q(\bar e)^kc_q( e)^lc_q(e)^{* {(n-l)}}c_q(e)^{*(n-k)}$}
\end{Lemma}

\begin{proof}
We use the Wick formula in Eq. \eqref{wick'}. Terms of the form $c_q(\bar e)^kc_q( e)^l c_q(e)^{* {(n-k)}}c_q(\bar e)^{*(n-l)}$ occur when we choose $\mathfrak{J}=J_1\cup J_2\subset\{1,\ldots,2n\}$ with $J_1\subset\{1,\ldots,n\}$ of cardinal $k$ and $J_2\subset \{n+1,\ldots,2n\}$ of cardinal $l$. Thus, the formula holds with $q_{k,l}=\sum_{J_1,J_2}q^{c(\mathfrak{J},\mathfrak{J}^c)}$, where the sum runs over all possibilities as mentioned.

Writing $\mathfrak{J}^c=K_1\cup K_2$ with 
$K_1=\mathfrak{J}^c\cap  \{1,\ldots,n\}$ and $K_2=\mathfrak{J}^c\cap  \{n+1,\ldots,2n\}$, we have 
$c(\mathfrak{J},\mathfrak{J}^c)=c(J_1,K_1)+c(J_2,K_2)+(n-k)l$, has all elements in $K_1$ cross those of $J_2$. Hence, 
$$q_{k,l}= q^{(n-k)l}. \sum_{J_1} q^{c(J_1,K_1)}. \sum_{J_2} q^{c(J_2,K_2)}.$$
It follows from \cite{B99} that $|q_{k,l}|\leq   |q|^{(n-k)l} C_q^2$. 
\end{proof}

As a result of Lemma \ref{ween} we have:
  
\begin{Lemma}\label{bnd2}
The sequence $((1-q)^{n}W({\bar e}^{ n} e^{  n}))$ is  bounded.	
\end{Lemma}

\begin{proof}
We rely on Lemma \ref{ween} and the triangle inequality with the help of 
Eq. \eqref{Normelt} and Eq. \eqref{Normelt2}. Note that 
\begin{eqnarray*}
\| (1-q)^{n}W({\bar e}^{ n} e^{  n})\|&\leq& \sum_{k,l=0}^n C_q^4 |q|^{(n-k)l}\sqrt{d_l d_{n-l}d_k d_{n-k}}\lambda ^{(k-l)/2} \\ 
&\leq& \sum_{k,l=0}^n C_q^6 |q|^{(n-k)l}\lambda ^{(k-l)/2}\\
& =& C_q^6\sum_{j=-n}^n  \lambda^{j/2} \sum_{\substack{0\leq k,l\leq n \\ k-l=j}} |q|^{(n-k)l}.
\end{eqnarray*}
If we denote by $c_{j,n}$ the last quantity, we have for $j<0$
$$c_{j,n}= \sum_{s=0}^{n+j} |q|^{(n-s)(s-j)}\leq (n+j+1) |q|^{-nj}.$$
If $j\geq 0$, then $$c_{j,n}= \sum_{k=j}^{n} |q|^{(n-k)(k-j)}\leq \frac 1 {1-|q|}.$$
Choose $n_0$ large so that $\frac{\abs{q}^n}{\sqrt{\lambda}}<1$ for all $n\geq n_0$. 
Combining the estimates
\begin{eqnarray*}
\| (1-q)^{n}W({\bar e}^{ n} e^{ n})\|&\leq& C_q^6\Big( \frac 1{1-|q|}\sum_{j=0}^n \lambda^{j/2} + \sum_{j=1}^n \lambda^{-j/2}(n+1) \abs{q}^{nj}\Big)\\ 
&\leq &C_q^6\Big( \frac 1{(1-|q|)(1-\sqrt{\lambda})} +\frac {(n+1) \frac{|q|^n}{\sqrt{\lambda}}}
{1-\frac{|q|^n}{\sqrt{\lambda}}}\Big),\quad \forall n\geq n_0.\end{eqnarray*}
This allows to conclude as $(n+1)\abs{q}^n\rightarrow 0$ as $n\rightarrow\infty$.
\end{proof}

We end this section with the following lemma. 

\begin{Lemma}\label{lim}
Let $ n_1, n_2  \in \mathbb N\cup\{0\}$ and 	$ \Psi \in \CH^{\otimes_q l }$ for some $l\geq0$. Then, 
\begin{align*}
&(1-q)^n c_q(e)^* (\bar{e}^{(n +n_1)} e^{ (n + n_2)}  \Psi) \rightarrow 0,\\
&(1-q)^{n/2} \lambda^{-n/4}c_q(e)^* (\bar{e}^{(n +n_1)} \Psi) \rightarrow 0, \text{ as } n\rightarrow\infty.
\end{align*}
\end{Lemma}

\begin{proof}
We only prove the first one, the second one is similar. Using Lemma \ref{SplitAdjoint} and Eq. \eqref{Leftmult}, it follows that 
\begin{align*}
&\norm{(1-q)^n c_q(e)^* \bar{e}^{(n +n_1)}   e^{   (n + n_2)}   \Psi }_q\\
=&(1-q)^n   \abs{q}^{(n_1 + n)}  \norm{ \bar{e}^{ (n +n_1)}  \big( c_q(e)^*  e^{   (n + n_2)}   \Psi \big)}_q\\
= & (1-q)^n   \abs{q}^{(n_1 + n)}\norm{ c_q(\bar{e})^{(n +n_1)} c_q(e)^* c_q(e^{n+n_2}) \Psi }_q\\
\leq &C_q^{\frac{3}{2}}  (1-q)^n   \abs{q}^{(n_1 + n)} \norm{\bar{e}^{( n+ n_1)}}_q  \norm{e^{( n+ n_2) }}_q\norm{e}_q\norm{\Psi}_q \quad (\text{by \eqref{Normelt}})\\
= &  C_q^{\frac{3}{2}}  (1-q)^n   \abs{q}^{(n_1 + n)} \lambda^{\frac{(n+n_1- n- n_2 -1)}{4} }\sqrt{d_{n+n_1} d_{n+n_2}d_1} (1-q)^{-\frac{2n+n_1+n_2+1}{2}} \norm{\Psi}_q\\
&\quad\quad\quad\quad\quad\quad\quad\quad\quad\quad\quad\quad\quad\quad\quad\quad\quad\quad\quad\quad\quad\quad\quad\quad (\text{by Eq. \eqref{Normelt2}})\\
\leq &  C_q^3 (1 -q)^{-\frac{ n_1+n_2 +1}{2}} \lambda^{\frac{(n_1- n_2 -1)}{4} }\abs{q}^{(n_1 + n)} \norm{\Psi}_q \xrightarrow{ n \rightarrow \infty} 0.
\end{align*}
\end{proof}

\subsection{Technicalities to treat all cases}\label{casebig3}

The results in this section are crucial for tracking the relative commutant of $M_q^\varphi$ and will be used even in the case when $dim(\CH_\R)\geq 3$. Thus, we adjust the set up of this section in such a way that it applies to all cases.

For the next three results, we assume $\CH_\R= \mathbb{R}^2\oplus \mathcal{K}_\R$, where $\mathcal{K}_\R$ is a real Hilbert space $($could be $0)$, and $\mathbb{R}^2$ is reducing subspace for $(U_t)$ with associated sub representation being ergodic i.e., $(\mathbb{R}^2, U_t)$ is as before. Define $\mathcal{O}=\{e,\bar e,\Omega\}\cup \{\xi\in \mathcal{K}_\C: W(\xi) \text{ is analytic for }(\sigma_t^\varphi)\}$.

The next lemma tracks $w^*$-limits of certain sequences of operators, which in turn tracks the relative commutant of the centralizers.

\begin{Lemma}\label{comp}
For any $\chi \in \CH^{\otimes_q i}$ and $\eta\in \CH^{\otimes_q j}$, $a,\, b,\, \alpha,\, \beta\in \mathbb Z$, we have 
\begin{align*}
&\lim_{n\to \infty} (1-q)^{2n}\langle \eta \bar{e}^{n+b }e^{n+\beta},\,\bar{e}^{n+a}e^{n+\alpha}\chi
\rangle_q\\
=&\,\delta_{\substack{a=b+j\\\alpha+i=\beta}}{d_\infty^2}\, \frac {\lambda^{(a-\beta)/2}}{(1-q)^{a+ \beta}}
\langle \eta,\,\frac{\bar{e}^j}{\|\bar{e}^j\|_q^2}\rangle_q \langle \frac{{e}^i}{\|e^i\|_q^2},\,\chi\rangle_q.
\end{align*} 
\end{Lemma}

\begin{proof}
As one is taking limit as $n\rightarrow \infty$, one can assume that
$n+a$, $n+\alpha$, $n+b$ and $n+\beta$ are all positive and large. By
linearity, we can assume that $\eta$ and $\chi$ are elementary tensors
in the letters from $\mathcal{O}$. If $n+a\geq n+\alpha$, then
$a-\alpha\geq 0$. Consequently, $(1-q)^{n}
\bar{e}^{n+a}e^{n+\alpha}\chi = c_{q,r}(\chi) c_q(\bar
e^{a-\alpha})(1-q)^nW(\bar e^{n+\alpha} e^{n+\alpha})\Omega$ is a
bounded sequence from Eq. \eqref{Normelt2} and Lemma
\ref{bnd2}. Dealing with the case $n+a< n+\alpha$ similarly (replacing $\chi$ by $e^{\alpha-a}\chi$), it
follows that $(1-q)^{n}\bar{e}^{n+a}e^{n+\alpha}\chi$ is
bounded. Similarly, $(1-q)^{n}\eta\bar{e}^{n+b}e^{n+\beta}$ is
bounded.

First, assume that $\eta$ or $\chi$ contains a letter different from $e, \bar e, \Omega$. For simplicity assume $\eta=\eta_1\cdots\eta_j$  is such that at least one letter in $\eta \not\in \{e, \bar e, \Omega\}$. Let $t=\min\{l: \eta_l\not\in \{e,\bar e,\Omega\}, \, 1\leq l\,\leq j\}$. Let $T=\prod_{l=1}^{t-1} c_q(\eta_l)$. By Lemma \ref{SplitAdjoint}, it follows that there exists a finite set $F$ $($depending on $l)$ and scalars $c_{f,n}$ with $\sup_n\abs{c_{f,n}}<\infty$ for all $f\in F$, $a_f, \alpha_f\in \mathbb{Z}$ and vectors $\chi_f\in \mathcal{F}_q(\CH)$ for all $f\in F$ such that 
\begin{align*}
T^* \bar{e}^{n+a}e^{n+\alpha}\chi =\sum_{f\in F} c_{f,n} \bar{e}^{n+a_f}e^{n+\alpha_f}\chi_f.
\end{align*}
Fix $f\in F$. From Lemma \ref{SplitAdjoint} again, it follows that as $n\rightarrow\infty$, 
\begin{align*}
(1-q)^nc_q(\eta_t)^*\bar{e}^{n+a_f}e^{n+\alpha_f}\chi_f = q^{2n+a_f+\alpha_f} (1-q)^n\bar{e}^{n+a_f}e^{n+\alpha_f} c_q(\eta_t)^*\chi_f\rightarrow 0.
\end{align*}
Summing over $f\in F$, it follows that the limit in the statement is $0$ and matches with the right hand side of the statement.  Arguing similarly with $\chi$ and using right creation operators, it follows that it is sufficient to assume that the letters in $\eta,\chi$ are all in $\{e,\bar e, \Omega\}$, otherwise both sides in the statement are $0$.

We do the proof by induction on $i+j$. Note that $\langle \bar {e}^j,\,\frac{\bar{e}^j}{\|\bar{e}^j\|_q^2}\rangle_q=1$ and similarly for $e$. So, if $\eta$ is an elementary tensor in the letter $e$ and $\bar e$, the quantity $ \langle \eta,\,\frac{\bar{e}^j}{\|\bar{e}^j\|_q^2}\rangle_q$ is 1 if $\eta=\bar{e}^j$ and 0 otherwise. Thus, counting the number of letters $e$ and $\bar e$ gives the Dirac condition.

We start with $i+j=0$. The Dirac condition is clear and from Eq. \eqref{Normelt2} we have 
\begin{align}\label{diffuse}
\|\bar{e}^{n+a}e^{n+\beta}\|_q^2= \lambda^{(a-\beta)/2}[n+a]_q![n+\beta]_q!=
\lambda^{(a-\beta)/2} (1-q)^{-2n-a-\beta}d_{n+a}d_{n+\beta}.
\end{align}
Thus, the limit in the statement is $\lambda^{(a-\beta)/2}(1-q)^{-(a+\beta)} d_\infty^2$.

Assume for instance $j>0$.   Let $\eta=e_1\cdots e_j$ with $e_k\in\{ e, \bar{e}\} $. We have
\begin{align}\label{Limitinindunction}
(1-q)^{2n}\langle \eta \bar{e}^{n+b}e^{n+\beta},\,\bar{e}^{n+a}e^{n+\alpha}\chi\rangle_q=(1-q)^{2n}\langle \eta^\prime \bar{e}^{n+b}e^{n+\beta},c_q(e_1)^*\bar{e}^{n+a}e^{n+\alpha}\chi\rangle_q, 
\end{align}
with $\eta^\prime=e_2\cdots e_j$. By making arguments as in the first paragraph of the proof, it follows that 
$(1-q)^n\eta'\bar{e}^{n+b }e^{n+\beta}$ is bounded. Thus, using Lemma  \ref{lim}, we get that the limit in Eq. \eqref{Limitinindunction} is zero unless $e_1=\bar{e}$. If this is so, then from Lemma \ref{SplitAdjoint} we have
\begin{align*}
c_q(e_1)^*\bar{e}^{n+a}e^{n+\alpha}\chi=\lambda^{1/2}[n+a]_q \bar{e}^{n+a-1}e^{n+\alpha}\chi+ q^{2n+a+\alpha}\bar{e}^{n+a}e^{n+\alpha}\big(c_q(e_1)^*\chi\big). 
\end{align*}
Thus, by induction the limit exists and 
$$\lim_n (1-q)^{2n}\langle \eta \bar{e}^{n+b}e^{n+\beta},\bar{e}^{n+a}e^{n+\alpha}\chi\rangle_q= \frac{\lambda^{1/2}}{1-q}\lim_n(1-q)^{2n}\langle \eta^\prime \bar{e}^{n+b }e^{n+\beta},\bar{e}^{n+a-1}e^{n+\alpha}\chi
\rangle_q.$$
One can argue in the same way if $i>0$ using right creation operators and Eq. \eqref{Rightmultadj} when dealing with $\chi$. 
\end{proof}

The next lemma is the key to factoriality and irreducible centralizers. 

\begin{Lemma}\label{sxi}
The $w^*$-limit of $(1-q)^{2n}	W_r(\bar{e}^n e^n) W(\bar{e}^n e^n)$ exists and is the positive rank-one operator $T_\xi$ where 
\begin{align*} 
\xi = S_\infty^*(\Omega)=\|\cdot\|_q\text{-}\lim_{n\to \infty}(1-q)^n\lambda^{n/2} W(\bar e^n) e^n=d_\infty \sum_{k=0}^\infty c_k^2(1-q)^k \lambda^{k/2} {\bar e}^k e^k.
% d_\infty S^*(\Omega)=\|\cdot\|_q-\lim_{n\to \infty}(1-q)^n\lambda^{n/2} W(\bar e^n) e^n.
\end{align*}
Needless to say, $T_\xi$ is a scalar multiple of the rank-one projection $P_{\frac{\xi}{\norm{\xi}_q}}$.
\end{Lemma}

\begin{proof}
First note that $e, \bar e$ are analytic vectors for $(\Delta^{it}_\varphi)$, thus words in them belong to the Tomita algebra associated to $\varphi$. Further, $\bar{e}^{ n}e^{ n}\in M_q^\varphi\Omega$ and $J(\bar{e}^ne^n)=\bar{e}^ne^n$ for all $n$ $($see Thm. \ref{CentraliserDescribe}$)$. Consequently, $z_n=(1-q)^{2n} W_r(\bar{e}^{ n}e^{ n}) W(\bar{e}^{ n} e^{  n})\in \mathbf{B}(\mathcal{F}_q(\mathcal{H}))$ is a bounded sequence; also $\xi_n=(1-q)^{n}  \bar{e}^n e^n$ is bounded in $n$ $($by Lemma \ref{bnd2}$)$.

To show the $w^*$-convergence, we just need to find the limit of $\langle z_n \Phi, \Psi \rangle_q$ where $ \Psi \in \CH^{\otimes_q l}$ and $ \Phi \in \CH^{\otimes_q k}$, $k,l\geq 0$.  By density, we can also assume  that $\Phi =e_1 e_2 \cdots e_k$ and $\Psi=f_1 f_2 \cdots  f_l$ where $ e_i, f_j\in \mathcal{O}$. Then, we  have 
\begin{align*}
\langle \Psi,\,z_n \Phi \rangle_q  &= \langle \Psi,~~(1-q)^{2n}	W_r(\bar{e}^n e^n) W(\bar{e}^n e^n)   \Phi  \rangle_q \\	
&= (1-q)^{2n} \langle W(\Psi) \bar{e}^n e^n,~~W_r(\Phi) \bar{e}^n e^n \rangle_q.
\end{align*}	 

By the Wick formulas $($Prop. \ref{wick}$)$,  
$$W(\Psi)=\sum_{ 0\leq j\leq l,\,\sigma\in S_{l,\,j} } q^{| \sigma|}u_{\sigma,\,j}\,\text{ if }l>0, \quad
W_r(\Phi)=\sum_{ 0\leq i\leq k,\, \rho\in S_{k,\,i}} q^{|{\rm flip}\,\circ\,\rho|}v_{\rho,\,i}\text{ if }k>0, $$
where
\begin{align*} &u_{\sigma,\,j}=c_q(f_{\sigma(1)})\cdots c_q({f_{\sigma(j)}})c_q(\overline{f_{\sigma(j+1)}})^*c_q(\overline{f_{\sigma(l)}})^* \\
&\textrm{and }\\
&v_{\rho,\,i}=c_{q,r}(e_{\rho(1)})\cdots  c_{q,r}({e_{\rho(i)}})c_{q,r}(
\overline{g_{\rho(i+1)}}^r)^*c_{q,r}(\overline{g_{\rho(k)}}^r)^*,
\end{align*}
where $g_{\rho(i+1)}, \cdots, g_{\rho(k)}\in \CH_\C$ are vectors that correct the difference between the left and right conjugations $($by $e_{\rho(i+1)}, \cdots, e_{\rho(k)}$ respectively$)$. Note that $W(\Psi)=1$ if $l=0$ and $W_r(\Phi)=1$ if $k=0$.

Let $\min(k,l)>0$. One observes that the left or right annhilation operators in a symbol other  than $e,\bar e$ does not contribute in the inner product $\langle \Psi,\,z_n \Phi\rangle_q$. Thus, the contributing factor in $\langle \Psi,\,z_n \Phi \rangle_q$ comprises of two scenarios: $(i)$ when a generic term in the Wick expansion formula of $W_r(\Phi)$ and $W(\Psi)$ both consists of creation operators only, $(ii)$ when the annihilation operators in a generic term in the Wick expansion formula of $W_r(\Phi)$ or $W(\Psi)$ consists only of the symbols $e$ or $\bar e$. 

For both cases, if the letters in the creation operators $($either left or right or both, as the case may be$)$ consist of a symbol different from $e$ or $\bar e$, then by Lemma \ref{comp}, the associated limit contributing to $\langle  \Psi ,\,z_n \Phi\rangle_q$ goes to $0$ as $n\rightarrow \infty$. Consequently, we can assume that $e_i, f_j\in \{e,\bar e\}$. 
The same conclusion holds when $\min(k,l)=0$ and $\max(k,l)>0$. It is clear that, this reduction is nothing but compressing $M_q$ by the Jones' projection onto its subalgebra $\Gamma_q(\R^2, U_t)^{\prime\prime}$ $($which possess $\varphi$-preserving conditional expectation$)$ to reduce to the set up when $\dim(\CH_\R)=2$.

Therefore, we can now assume that $W_r(\Phi)=\sum_{ 0\leq i\leq k,\, \rho\in S_{k,\,i}} q^{|{\rm flip}\,\circ\,\rho|}\lambda^{n_{k,i}}v_{\rho,\,i}$ if $k>0$, where $n_{k,i}$ is an integer to correct the difference between the left and right conjugations and the right annihilation operators in $v_{\rho,\,i}$ consists of symbols from $\{e,\bar e\}$.     

If $l>0$, then using Lemma \ref{lim} $(l-j)$ times and Lemma \ref{SplitAdjoint}, we get that $u_{\sigma,\,j}(\xi_n)$ goes to $0$ in $\norm{\cdot}_q$ unless $f_{\sigma(j+1)}=\cdots=f_{\sigma(l)}=e$. If this is so, then using Lemma \ref{SplitAdjoint} and Lemma \ref{Adjpowers} we have: $u_{\sigma,\,j}(\xi_n)=\frac{[n]_q!}{[n-l+j]_q!}\lambda^{(l-j)/2}(1-q)^{n}\Psi_j\bar{e}^{n-l+j}e^n$, where $\Psi_j=f_{\sigma(1)}f_{\sigma(2)}\cdots f_{\sigma(j)}$. Setting $\Psi^c_j=f_{\sigma(j+1)}\cdots f_{\sigma(l)}$, we may rewrite $u_{\sigma,\,j}(\xi_n)$ in full generality as 
$$u_{\sigma,\,j}(\xi_n)=  \frac{[n]_q!}{[n-l+j]_q!}\lambda^{(l-j)/2}(1-q)^{n}
\langle \Psi_j^c,\,\frac {e^{l-j}}{\|{e}^{l-j}\|_q^2} \rangle_q\Psi_j \bar{e}^{n-l+j}e^n.$$

We can do the same for $\Phi$ to get $v_{\rho,\,i}(\xi_n)=\frac{[n]_q!}{[n-k+i]_q!}\lambda^{(i-k)/2}(1-q)^{n}\bar{e}^{n}e^{n-k+i} e_{\rho(i)}\cdots e_{\rho(1)}$  provided that $e_{\rho(i+1)}=\cdots=e_{\rho(k)}=\bar e$ if $k>0$. Set $\Phi_i= e_{\rho(i)}\cdots e_{\rho(1)}$ and $\Phi_i^c= e_{\rho(k)}\cdots e_{\rho(i+1)}$. Then,
$$v_{\rho,\,i}(\xi_n)=  \frac{[n]_q!}{[n-k+i]_q!}\lambda^{(i-k)/2}(1-q)^{n}\langle \frac{\bar{e}^{k-i}}{\norm{\bar{e}^{k-i}}_q^2},\,\Phi_i^c\rangle_q\,
\bar{e}^{n}e^{n-k+i}\Phi_i.$$	 

We are in position of using Lemma \ref{comp} with $a=\beta=0$ to get the existence of 
\begin{align}\label{Wicklim}
&\lim_n \langle u_{\sigma,\,j}(\xi_n),\,v_{\rho,\,i}(\xi_n)\rangle_q\\
%\nonumber=\,& \delta_{\substack{2j=l\\2i=k}} \frac{d_\infty^2 \lambda^{(i+j)/2}}{(1-q)^{i+j}}
%\langle\frac {\bar{e}^{j}}{\|\bar{e}^{j}\|_q^2},\Psi_j\rangle_q \langle \frac {e^{\textcolor{red}{l-}j}}{\|%{e}^{\textcolor{red}{l-}j}\|_q^2},\Psi_j^c \rangle_q\langle\Phi_i,\frac {{e}^{i}}{\|e^{i}\|_q^2}\rangle_q\langle\Phi_i^c,%\textcolor{red}{\frac {{\bar{e}}^{k-i}}{\|{\bar{e}}^{k-i}\|_q^2}}\rangle_q\\
%
\nonumber=\,& \delta_{\substack{2j=l\\2i=k}} \frac{d_\infty^2 \lambda^{(i+j)/2}}{(1-q)^{i+j}}
\langle\Psi_j,\,\frac {\bar{e}^{j}}{\|\bar{e}^{j}\|_q^2}\rangle_q \langle\Psi_j^c,\, \frac {e^{j}}{\|{e}^{j}\|_q^2} \rangle_q\langle\frac {{e}^{i}}{\|e^{i}\|_q^2},\,\Phi_i\rangle_q\langle{\frac {{\bar{e}}^{i}}{\|{\bar{e}}^{i}\|_q^2}},\,\Phi_i^c\rangle_q, 
\end{align}
{when $k,l>0$}. This can be interpreted as decoupled scalar products. {When $k$ or $l$ is $0$, there is no Wick product expansion in terms of creation and annihilation operators, but the obvious modifications justify the existence of the desired limit(s)  using Lemma \ref{comp} and matches Eq. \eqref{Wicklim}}. 

Thus, summing all the terms we get that for some $\chi_k\in \CH^{\otimes_q k}$ and 
$\eta_l\in \CH^{\otimes_q l}$,
$$\lim_n \langle \Psi,\, z_n \Phi \rangle_q = \langle \chi_k,\,\Phi\rangle_q.\langle  \Psi,\,\eta_l\rangle_q,\quad \forall\, k,l\geq 0.$$

Now we will use some arguments to avoid the use of $q$-symetrization operators to identify the vectors. Since $z_n$ is bounded, this justifies that the $w^*$-limit exists and has rank 1. As $z_n=z_n^*$, we have $\chi_l=\eta_l$ for all $l\geq 0$. Taking $\Phi=\Omega$, we have that $z_n\Omega$ is weakly converging to {$\zeta=\langle \chi_0,\Omega\rangle_q\big(\oplus_{m=0}^\infty\chi_m\big)$}. 

We want to identify $\xi = \oplus_{m=0}^\infty \chi_m$ as $z_n\overset{w.o.t.}\rightarrow T_{\xi}$ $($the rank-one limit$)$, and to do so we identify $\zeta$. To find $\zeta$, we want to consider $\lambda^{-n/2}(1-q)^n\langle W(\Psi) \bar{e}^n, \,\bar{e}^n\rangle_q$. This is a bounded sequence $($see Lemma \ref{bnd1}$)$.  

Note that if $l=0$ (i.e., $\Psi=\Omega$), then $\lim_n\lambda^{-n/2}(1-q)^n\langle W(\Psi) \bar{e}^n,\,\bar{e}^n\rangle_q=d_\infty$. 
Let $l>0$. As above, we can use the Wick formula for $\Psi$ and the situation as in Lemma \ref{comp}. We get that $u_{\sigma,\,j}(\lambda^{-n/4}\bar{e}^n)$ is $0$ unless $f_{\sigma(j+1)}=\cdots=f_{\sigma(l)}=e$. If this is so, $u_{\sigma,\,j}(\lambda^{-n/4}\bar{e}^n)=\frac{[n]_q!}{[n-l+j]_q!}\lambda^{(l-j)/2-n/4}\Psi_j\bar{e}^{n-l+j}$ $($Lemma \ref{Adjpowers}$)$. Taking scalar product with $\lambda^{-n/4}\bar{e}^n$, recalling that $\|\bar{e}^n\|_q^2= [n]_q!\lambda^{n/2}$ and computing again as in Lemma \ref{comp}, we have 
$$\lim_n \lambda^{-n/2}(1-q)^n\langle u_{\sigma,\,j}(\bar{e}^n), \,\bar{e}^n\rangle_q=
\delta_{2j=l}\, d_\infty\lambda^{j/2}(1-q)^{-j}\langle \Psi_j,\,\frac {\bar{e}^{j}}{\|\bar{e}^{j}\|_q^2}\rangle_q\langle \Psi_j^c,\, \frac {e^{j}}{\|{e}^{j}\|_q^2} \rangle_q.$$

Thus, comparing with Eq. \eqref{Wicklim} $($setting $k=0)$ we have 
$$\lim_n \langle \Psi ,\,z_n \Omega\rangle_q=d_\infty \lim_n
\lambda^{-n/2}(1-q)^n\langle W(\Psi) \bar{e}^n,\,\bar{e}^n\rangle_q.$$ 
But using Eq. \eqref{modulartheory}, we have
$$\langle W(\Psi) \bar{e}^n,\,\bar{e}^n\rangle_q=\langle \Psi,\,W(\bar{e}^n) W_r(\bar{e}^n)^*\Omega\rangle_q=\lambda^n\langle \Psi,\,W(\bar{e}^n) W(e^n)\Omega\rangle_q.$$

Since $S_n^*(\Omega)=\lambda^{n/2}(1-q)^n W(\bar{e}^n) W(e^n)\Omega$ $($use Lemma \ref{wen} or Eq. \eqref{modulartheory}$)$, we get that $\zeta=d_\infty S^*(\Omega)$. To compute its value, one can use Lemma \ref{Adjpowers} and Lemma \ref{wen} to derive 
\begin{align*}
\lambda^{n/2}(1-q)^nW(\bar e ^n)e^n=\sum_{k=0}^n (1-q)^k\frac{d_n^2}{d_{n-k}d_k^2}\lambda^{k/2} \bar e^k e^k.
\end{align*}
By taking the limit in $n$ and using the dominated convergence theorem, we have $$\zeta = d_\infty^2 \sum_{k=0}^\infty c_k^2(1-q)^k \lambda^{k/2} {\bar e}^k e^k,$$
and this expression tallies with the expression in $(2)$ of Rem. \ref{Rem1}.	

Comparing the two expressions for $\zeta$ we have $\langle\chi_0,\Omega\rangle_q\chi_0 = d_\infty^2\Omega$. Since $\chi_0 = \kappa\Omega$ where $\kappa\in \C$, it follows that $\abs{\kappa}^2=d_\infty^2$ $($scalar product is linear on the right$)$, i.e., $\abs{\kappa}=d_\infty$. Consequently,
\begin{align*}
\xi = \frac{1}{\langle \chi_0 , \Omega\rangle_q} \zeta =\frac{1}{\bar \kappa}d_\infty^2 \sum_{k=0}^\infty c_k^2(1-q)^k \lambda^{k/2} {\bar e}^k e^k = e^{-i \arg{(\kappa)}} d_\infty \sum_{k=0}^\infty c_k^2(1-q)^k \lambda^{k/2} {\bar e}^k e^k.
\end{align*} 
Finally, note that a rank-one self-adjoint operator is uniquely determined upto a phase factor of the associated vector. Also note that $\xi$ is not an unit vector, thus $T_\xi$ is a positive scalar multiple of $P_{\frac{\xi}{\norm{\xi}_q}}$. 
\end{proof}

%\textcolor{red}{$$\lambda^{n/2}(1-q)^nW(\bar e ^n)e^n=\sum_{k=0}^n (1-q)^k\frac{d_n^2}{d_{n-k}^2d_k}\lambda^{k/2} \bar e^k e^k.$$}

%\begin{Lemma}\label{car}
%With $\xi$ as in Lemma \ref{comp}, let $P$ be the orthogonal projection onto $[M_q M_q^\prime\xi ]^\perp$. Then, $M_q$ is a factor if and only if $P=0$.
%\end{Lemma}
%
%\begin{proof}
%It is clear that $P\in \mathcal Z(M_q)$. Thus, if $M_q$ is a factor then $P=0$ as $\xi\neq0$.
%
%Conversely, by Lemma \ref{sxi} $P_\xi\in M_q\vee M_q^\prime$. Let $x\in
%\mathcal Z(M_q)$ so that $xP_\xi=P_\xi x$. In particular, 
%$x(\xi)=\lambda \xi$ for some $\lambda\in \C$. Then, for any $m\in
%M_q$ and $n\in M_q^\prime$, $x(mn(\xi))=\lambda mn(\xi)$.
%If $P=0$, this means that $x(\eta)=\lambda \eta$ for all $\eta\in\mathcal{F}_q(\CH)$ and thus $x=\lambda 1$.
%\end{proof}

%
%For $\mathcal C(\mathcal M)$, this is a small variation.
%Assume now that $p$ is a central projection  in $\mathcal C(\mathcal M)$. 
%As $ (1-q)^{2n}	W_r(\bar{e}^n e^n) W(\bar{e}^n e^n)\in\mathcal C(\mathcal M)'\mathcal C(\mathcal M)$, we get that $P_\xi$ commutes with $p$. By replacing $p$ by $1-p$ we may assume $p(\xi)=0$. But then as $W(e^n)^*W(e^n)\in \mathcal C(\mathcal M)$ we still get that $y_n=\lambda^{n/4}(1-q)^{n/2}W(e^n)p(\Omega)$ and $S_n(p(\Omega))$ go to 0 and thus $p=0$. Thus $p=0$ or $1-p=0$ and $\mathcal C(\mathcal M)$ is a factor.

The next lemma will be used particularly when $\dim(\CH_\R)\geq 3$. It describes some amount of mixing available when $\dim(\CH_\R)$ is `large', which frees one from any bargain with the parameters $\lambda$ and $q$ to decide the  factoriality unlike the case when $\dim(\CH_\R)=2$.

\begin{Lemma}\label{weakly-1}
Suppose $\text{dim}(\CH_\R) \geq 3$ with $(U_t)$ as in the set up. Let $0\neq \eta \in \CH_\R $ be such that $ \eta \perp \bar{e}, e$ in $\langle\cdot,\cdot\rangle_q$. Further, let $\{ u_n\}$ be a sequence of unitaries in $vN(W(\eta))$ converging to $0$ in the w.o.t. Then, $u_n W(\xi) u_n^* \rightarrow c 1$ in the $w.o.t$ for some $c>0$, where $\xi=d_\infty \sum_{k=0}^\infty c_k^2(1-q)^k \lambda^{k/2} {\bar e}^k e^k$.
\end{Lemma}

\begin{proof}
Firstly, note that $vN(W(\eta))$ is a diffuse abelian von Neumann algebra \cite[\S4]{bik-kunal}. Thus, the desired sequence $\{u_n\}$ exists as in the statement. 

By Lemma \ref{sxi}, the series defining $\xi$ is convergent in $\mathcal{F}_q(\CH)$ in $\norm{\cdot}_q$. 
Further, $\{ (1-q)^k W(\bar{e}^k e^k )\}$ is a bounded sequence by Lemma \ref{bnd2}. Since $c_k$ is a bounded sequence, so $d_\infty\sum_{k=0}^N c_k^2(1-q)^k\lambda^{k/2} W(\bar{e}^k e^k)$ is convergent in norm as $N\rightarrow\infty$ in $M_q^\varphi$. Using Thm. \ref{CentraliserDescribe}, it follows that 
\begin{align*}
W(\xi) = d_\infty \sum_{k=0}^\infty c_k^2(1-q)^k\lambda^{k/2} W(\bar{e}^k e^k )\in M_q^\varphi.
\end{align*}
Therefore, it is sufficient to show that $(1-q)^ku_n W(\bar{e}^k e^k ) u_n^*$ converges to $0$ in the $w.o.t$ as $ n\rightarrow \infty$ for each $k \in\N$. Further, the boundedness of the sequence entails that it is sufficient to verify 
the convergence with vectors of the form $\Psi = e_1 e_2\cdots e_l\in \CH^{\otimes_q l}$ and $\Phi = f_1 f_2\cdots f_m\in \CH^{\otimes_q l^\prime}$ where $e_i, f_j \in \mathcal{O}$ and $l,l^\prime\geq 0$.

 Now, note that 
\begin{align*}
\langle  \Phi,\, u_n W(\bar{e}^k e^k ) u_n^* \Psi \rangle_q&=\langle u_n^* \Phi,\,W(\bar{e}^k e^k ) u_n^* \Psi  \rangle_q\\
&=\langle W_r(\Phi)u_n^* \Omega,\,W(\bar{e}^k e^k ) W_r(\Psi) u_n^* \Omega \rangle_q\\&=\langle W_r(\Psi)^*W_r(\Phi)u_n^* \Omega,\,W(\bar{e}^k e^k ) u_n^* \Omega \rangle_q.
\end{align*}
When $k>0$, this last quantity goes to 0 as $n\to \infty$ as
$\bar{e}^k e^k \bot \,span \{\eta^l:\ l\geq0\}$ $($see \cite[\S4]{bik-kunal}$)$. This a consequence of
the Wick formula and Lemma 1 in \cite{SW} as explained around (4.2)
there.

As $\text{span}_\C\mathcal{O}$ is dense in $\mathcal{F}_q(\CH)$  and $(1-q)^{k}W({\bar e}^k e^k)$ is uniformly bounded $($Lemma \ref{bnd2}$)$, it follows that $(1-q)^ku_n W(\bar{e}^k e^k ) u_n^*\rightarrow 0$ as $n\rightarrow \infty$ in the $w.o.t$ for each $k\geq 1$. Clearly $c=d_\infty$ and the proof is complete. 
\end{proof}

\section{Factoriality}\label{Secfactor}

%\subsection{ Factoriality of  $\Gamma(\CH_\R, U_t )^{\prime\prime}$  when dim$(\CH_{\R}) \geq 3$ and $(U_t)$ have a two dimensional ergodic subrepresentation}\label{casebig3}

This is the main section of this paper. Here we establish the factoriality of $M_q$. \textit{The set up of \S\ref{casebig3} will be in force in this section}. As informed earlier, we will have to compromise with the parameter $\lambda$ to assert the factoriality of $M_q$ when $\dim(\CH_\R)=2$. In this case though, we will directly deduce that $M_q^\varphi$ is irreducible. When $\dim(\CH_\R)\geq 3$, $M_q$ is a factor \textit{regardless} of the value of the parameter $\lambda$ defining $(U_t)$ on $\R^2$ and $-1<q<1$. However, deducing that $M_q^\varphi$ is irreducible will not be direct and will be constrained. When $q=0$, $M_0$ is the free Araki-Woods factor \cite{Shlyakhtenko} and there is nothing to prove in this case. 

As an outcome of the results in \S\ref{mainsec}, we have the following:

\begin{Lemma}\label{car}
With $\xi$ as in Lemma \ref{sxi}, let $P$ be the orthogonal projection onto $[M_q^\varphi M_q^\prime\xi ]^\perp$. Then, $(M_q^\varphi)^\prime\cap M_q$ is trivial if and only if $P=0$. In particular, $M_q^\varphi$ and $M_q$ are factors if and only if $P=0$. 
\end{Lemma}

\begin{proof}
It is clear that $P\in \mathcal (M_q^\varphi)^\prime\cap M_q$. Suppose that $(M_q^\varphi)^\prime\cap M_q=\C1$. Then, $P=0$ as $\xi\neq 0$ and $M_q^\varphi\text{ and } M_q$ are factors. 

Conversely, by Lemma \ref{sxi} it follows that $T_\xi\in M_q^\varphi\vee M_q^\prime$. Let $x\in
(M_q^\varphi)^\prime\cap M_q$ so that $xT_\xi=T_\xi x$. In particular, 
$x(\xi)=\lambda \xi$ for some $\lambda\in \C$. Then, for any $m\in
M_q^\varphi$ and $n\in M_q^\prime$, $x(mn(\xi))=\lambda mn(\xi)$.
If $P=0$, this means that $x(\eta)=\lambda \eta$ for all $\eta\in\mathcal{F}_q(\CH)$ and thus $x=\lambda 1$.
\end{proof}

Now we are in position to state the first theorem on factoriality. 

\begin{Theorem}\label{Factor2dim}
Let $q\neq 0$. If $\dim(\CH_\R)=2$ and $(U_t)$ is ergodic then $(M_q^\varphi)^\prime\cap M_q=\C1$ $($and hence $M_q^\varphi$ and $M_q$ are factors$)$ when 
\begin{equation*}
\lambda< \begin{cases}(1+\frac{C_q}{d_\infty})^{-2},\quad&\text{if }q>0;\\
(1+\frac{(C_q(1-q))^2D(q)}{d_\infty})^{-2},\quad &\text{if }q<0.\end{cases}
\end{equation*}
\end{Theorem}

\begin{proof}
Consider $P$ from Lemma \ref{car}. Then, we have $P(\xi)=0$. Thanks to Lemma \ref{sxi} and Eq. \eqref{modulartheory}, we have that 
\begin{eqnarray*}
0=\langle P(\xi),\, \Omega\rangle_q&=& \lim_n  \lambda^{n/2}(1-q)^n 
\langle P W(e^n)^*W(e^n)\Omega,\,\Omega\rangle_q\\&=&\lim_n  \lambda^{n/2}(1-q)^n 
\langle  W(e^n)^*W(e^n)P\Omega,\,P\Omega\rangle_q\\ &=& \lim_n  \lambda^{n/2}(1-q)^n \|W(e^n)P(\Omega)\|_q^2.
\end{eqnarray*}

Thus, $\eta_n=\lambda^{n/4}(1-q)^{n/2}W(e^n)P(\Omega)$ goes to 0 in $\norm{\cdot}_q$. Thanks to Lemma \ref{bnd1} $(i)$, we still get that $\lambda^{n/4}(1-q)^{n/2}c_q(e)^{*n}\eta_n= S_n(P(\Omega))$ goes to 0. Consequently, $S_\infty(P(\Omega))=0$. But $S_\infty$ is invertible by Lemma \ref{s} and the hypothesis, and hence $P(\Omega)=0$ forcing $P=0$. Now use Lemma \ref{car}.
\end{proof}

Now, we assume that $\CH_\R= \mathbb{R}^2\oplus \mathcal{K}_\R$ where $\mathcal{K}_\R$ is a non-zero real Hilbert space, and $\mathbb{R}^2$ is reducing subspace for $(U_t)$ with associated sub representation being ergodic i.e., $(\mathbb{R}^2, U_t)$ is as in \S\ref{casebig2}. It should be noted that unlike Thm. \ref{Factor2dim}, the invertibility of the operator $S_\infty$ has \textit{no role} to play in deciding the factoriality in this case.

\begin{Theorem}\label{gre3factor}
Suppose  $\text{dim}(\CH_\R) \geq 3$. Then $M_q$ is a factor. 	
\end{Theorem}

\begin{proof}
Let $P$ be the orthogonal projection onto $[M_q M_q^\prime\xi]^\perp$, where $\xi$ is as in Lemma  \ref{weakly-1}. Since $dim(\CH_\R)\geq 3$, choose $0\neq \eta\in \CH_\R$ such that $s_q(\eta)$ is analytic for $(\sigma_t^\varphi)$ and $\eta\perp \{e,\bar e\}$ in $\langle\cdot,\rangle_q$. Then, $u_n = e^{-in s_q(\eta)}$ is an analytic sequence of unitaries 
such that $u_n\rightarrow 0$ in the $w^*$-topology. The $w^*$-convergence holds as the spectral measure of $s_q(\eta)$ is Lebesgue absolutely continuous $($see \cite[\S4]{bik-kunal}$)$. By Lemma \ref{weakly-1}, it follows that 
\begin{align*}
J\big(\sigma_{\frac{i}{2}}^\varphi(u_n)\big)^*J u_n^* \xi=u_n^* W(\xi) u_n\Omega\rightarrow c\Omega \text{ weakly as }n\rightarrow\infty.
\end{align*}
It follows that $[M_q M_q^\prime\xi]$ is weakly dense in $\mathcal{F}_q(\CH)$, and hence norm dense in $\mathcal{F}_q(\CH)$ by a theorem of Mazur. It follows that $P=0$. 

From Lemma \ref{sxi}, we have $T_\xi\in M_q^\varphi\vee M_q^\prime$. Let $x\in \mathcal{Z}(M_q)$. Then, $xT_\xi = T_\xi x$. Replacing the role of $M_q^\varphi$ by $M_q$ in the proof of Lemma \ref{car} $($second paragraph$)$, the argument is verbatim. 
\end{proof}	

\begin{Remark}\label{opencase}
The following comments are in order. 
\begin{enumerate}
\item In view of the main results in \cite{bik-kunal,ER} and the results in this section, the factoriality problem for $\Gamma_q(\CH_\R, U_t)^{\prime\prime}$ remains open \text{only} in the case when $dim(\CH_\R)=2$, $(U_t)$ is ergodic and $\lambda$ is \textit{not small} in the sense of Thm. \ref{Factor2dim}.
\item When $dim(\CH_\R)=2$, our proof is different than the existing proofs of the factoriality of $M_q$ under various assumptions $($in \cite{bik-kunal,Hiai,ER,Shlyakhtenko, Ne15}$)$ because of the role of the operator $S_\infty$ in the proof, and this process will not work when $q=0$ and $\lambda>\frac 12$ (see Rem. \ref{Rem1}). 
\item Let $\eta\in\CH_\R$ be such that $\norm{\eta}_{U}= 1\,(=\norm{\eta}_q)$ and $\eta$ is not fixed by $(U_t)$.  Then, either $(U_t)$ has a non-zero weakly mixing component or has a two-dimensional ergodic sub representation characterized by $\lambda\in (0,1)$. If $dim(\CH_\R)\geq 3$ or $dim(\CH_\R)=2$, $q\neq 0$ and $\lambda$ is `small' in the sense of Thm. \ref{Factor2dim} or $dim(\CH_\R)=2$ and $q=0$, then $vN(W(\eta))\subseteq M_q$ is a split inclusion $($see \cite[Thm. 4.6]{BM2}$)$. 
\end{enumerate}
\end{Remark}

\section{Revisiting Centralizer}\label{cent_irreducibel}

The first result on $M_q^\varphi$ is Thm. 3.2 of \cite{Hiai}. In \S\ref{SectionCentralizer} and in \cite[\S7]{bik-kunal} we have discussed on the structure of $M_q^\varphi$. In this section, we continue our discussion on $M_q^\varphi$. Note that $M_q^\varphi$ depends on the almost periodic component of $(U_t)$ $($see Thm. \ref{CentraliserDescribe}$)$. Thm. 7.1 of \cite{bik-kunal} says that if there is a non-trivial fixed point of $(U_t)$ and the almost periodic component of $(U_t)$ is at least two-dimensional then $(M_q^\varphi)^\prime\cap M_q =\C1$. 

Consequently, when $\dim(\CH_\R)=2$ the last statement together with Thm. \ref{Factor2dim} provide all that is known about $M_q^\varphi$. When the almost periodic part of $(U_t)$ is three-dimensional or five-dimensional one can apply \cite[Thm. 7.1]{bik-kunal}. Now we proceed to describe the structure of $M_q^\varphi$ when the almost periodic part of $(U_t)$ is sufficiently large. 

\begin{Theorem}\label{trivial-rel-comm}
Let $\CH_\R^{ap}\subseteq \CH_\R$ denote the almost periodic part of $(U_t)$. If $dim(\CH_\R^{ap})\geq 5$, then  $(M_q^\varphi)^\prime\cap M_q=\C1$. In particular, $M_q^\varphi$ is a factor.
\end{Theorem}

\begin{proof}
If $(U_t)$ has a non-zero fixed point then there is nothing to prove. So we only have to prove in the case when $\dim(\CH_\R^{ap})\geq 6$ and $(U_t)$ is ergodic. In this case, $\CH_\R= \oplus_{i=1}^3\R^2_i \oplus \mathcal{K}_\R$, where $\R^2_i :=\R^2$ for $1\leq i\leq 3$ and $\mathcal{K}_\R$ are invariant subspaces of $(U_t)$.  

Denote $M_q^{(1)} =\Gamma_q(\R^2_1, {U_t}_{\upharpoonleft \R^2_1})^{\prime\prime}$, $M_q^{(2)}= \Gamma_q(\R^2_2\oplus \R_3^2, {U_t}_{\upharpoonleft \R^2_2\oplus \R_3^2})^{\prime\prime}$. Then $M_q^{(1)}, M_q^{(2)}$ are unital subalgebras of $M_q$ via the $q$-Gaussian functoriality \cite{Hiai}. Note that $M_q^{(1)}$ and $M_q^{(2)}$ posses $\varphi$-preserving conditional expectations by Takesaki's theorem \cite{Ta}. 

By Thm. \ref{gre3factor}, it follows that $M_q$ and $M_q^{(2)}$ are both factors. We claim that both $M_q$ and $M_q^{(2)}$ cannot be of type $\rm{I}$ and hence they are diffuse. Let us first assume this claim and finish the proof. Since $({U_t}_{\upharpoonleft \R^2_2\oplus \R_3^2})$ is almost periodic, $(M_q^{(2)})^{\varphi_{\upharpoonleft M_q^{(2)}}}$ is also diffuse \cite[Thm. 7.9]{DM20}. Let $u_n\in (M_q^{(2)})^{\varphi_{\upharpoonleft M_q^{(2)} }}$ be a sequence of unitaries that go to $0$ in the $w.o.t$. By Thm. \ref{CentraliserDescribe}, it follows that $u_n\in M_q^\varphi$ as well. 

Let $\xi$ be as in Lemma \ref{weakly-1} obtained by regarding $M_q^{(1)}$ as the von Neumann algebra in \S\ref{casebig2}. 
From Lemma \ref{weakly-1}, it follows that $u_n^* W(\xi) u_n\overset{w.o.t}\rightarrow c1$, where $c>0$ is a scalar. It follows that $[M_q^\varphi M_q^\prime \xi]=\mathcal{F}_q(\CH)$. Thanks to Lemma \ref{car}, the argument is complete.

Now it remains to establish the claim. The proof of the claim is identical for both $M_q^{(2)}$ and $M_q$. So we prove it for $M_q$ and work with the sub representation $(\R_1^2, {U_t}_{\upharpoonleft \R_1^2})$ as in \S\ref{casebig3}.

Recall that a factor $M$ equipped with a faithful normal state $\psi$ is of type $\rm{I}$ if and only if, $M\ni x\mapsto \Delta_\psi^{\frac{1}{4}}x\Omega_\psi\in L^2(M,\psi)$ is a compact embedding $($see \cite[Cor. 2.9]{BDL} or \cite[Thm. 7.13]{DM20}$)$. 
Note that $\{(1-q)^{n}W({\bar e}^ne^n)\}_{n\geq 1}\subseteq M_q^\varphi$ is bounded from Lemma \ref{bnd2}. However, $\Delta^{\frac{1}{4}}(1-q)^{n}W({\bar e}^ne^n)\Omega= (1-q)^n{\bar e}^ne^n$, and this sequence has no converging subsequence in $ \CF_q(\CH) $. Therefore, the symmetric embedding of $M_q$ in $\mathcal{F}_q(\CH)$ is not compact proving the claim.
\end{proof}

\begin{Remark}\label{cent}
$ $
\begin{enumerate}
\item If $dim(\CH_\R^{ap})=4$ and $(U_t)$ is ergodic, then upon assuming that the parameter $\lambda\in(0,1)$ which defines a two-dimensional sub representation is small in the sense of Thm. \ref{Factor2dim}, one can still conclude that $(M_q^\varphi)^\prime\cap M_q=\C1$ $($when $q\neq 0)$. 
\item In all cases where $(M_q^\varphi)^\prime\cap M_q=\C1$, the $S$-invariant of Connes can be directly deduced from the modular theory of the vacuum state and is exactly as in \cite[Thm. 8.2]{bik-kunal}. 
\item Note that the last part of the proof of Thm. \ref{trivial-rel-comm} along with Thm. 5.2 of \cite{DM2} says that if $\CH^{ap}\neq 0$, then $M_q^{\varphi}$ cannot be discrete $($direct sum of matrix algebras$)$. 
\item Suppose that $dim(\CH_\R)=4$ and $(U_t)$ is ergodic. In this case if $\zeta\in \CH_\R$ be such that $\norm{\zeta}_q=1$, then $M_\zeta=vN(s_q(\zeta))\subseteq M_q$ is a quasi-split inclusion $($see \cite[Thm. 4.6]{BM2}$)$. Since $M_q$ is a factor $($Thm. \ref{gre3factor}$)$, $M_\zeta\otimes 1\subseteq M_q\otimes \mathbf{B}(\mathcal{K})$ is a split inclusion $($see \cite[Thm. 3.8]{BM2}$)$, where $\dim(\mathcal{K})=\aleph_0$. It follows that there is no normal conditional expectation from $M_q$ onto $M_\zeta$. This forces that $M_q$ cannot be a $\rm{II}_1$ factor, for if it were there would have been a normal conditional expectation from $M_q$ onto $M_\zeta$ preserving the canonical trace.
\end{enumerate}
\end{Remark}

%\begin{proof}
%When $dim(\CH_\R)=4$ and $(U_t)$ is ergodic, then $M_q$ is a factor from Thm. \ref{gre3factor} and the last part of the proof of Thm. \ref{trivial-rel-comm} shows that $M_q$ is a continuous factor. If $M_q$ is semifinite, then by Thm. 7.8 of \cite{DM20} there exists a MASA $\mathcal{A}$ of $M_q$ inside the centralizer $M_q^\varphi$. Then,  $\mathcal{Z}(M_q^\varphi)\subseteq (M_q^\varphi)^\prime\cap M_q$. Consequently, 
%\begin{align*}
%aP_{\xi_0} = P_{\xi_0}a \quad\forall a\in \mathcal{Z}(M_q^\varphi)
%\end{align*}
%
%
%\end{proof}

In view of Rem. \ref{cent}, we are still not in position to compute the $S$-invariant when $\dim(\CH_\R)=4$ and $(U_t)$ is ergodic. Therefore, we conclude this section by proving that $M_q$ is never semifinite. 

\begin{Theorem}\label{semi}
$M_q$ is not semifinite. In particular, if $dim(\CH_\R)=4$ then $M_q$ is a type $\rm{III}$ factor. 
\end{Theorem}

\begin{proof}
If $(U_t)$ has a fixed (non-zero) vector or a non-trivial weakly mixing component, the statement follows from \cite{bik-kunal}. 
Therefore, we assume that $(\CH_\R,U_t)$ is as in the set up of \S\ref{casebig2}. 

We claim that $(\sigma_t^\varphi)$ cannot be inner. Suppose there exists a one-parameter group of unitaries $(u_t)\subseteq M_q^\varphi$ such that $\sigma_t^\varphi(x) =u_tx u_t^*$ for all $x\in M_q$ and $t\in \R$. 
Then, by Lemma \ref{sxi}, we have $u_t T_{\xi} = T_{\xi} u_t$ for all $t\in \R$. Therefore, there exists scalars $z_t$ with $\abs{z_t}=1$ such that $u_t\xi =z_t\xi$ for all $t$. Replacing $u_t$ by $\overline{z_t}u_t$, we may assume that $z_t=1$ for all $t$. Then, $u_tW(\xi)= W(\xi)$ for all $t$. 

Let $\mathcal{K} = Ker(W(\xi))$ and $\mathcal{L}= \mathcal{K}^\perp$. Then, ${u_t}_{\upharpoonleft \mathcal{L}}=1_{\mathcal{L}}$. Put $w_n=\abs{W\big((\lambda^{\frac{1}{4}}(1-q)^{\frac{1}{2}}e)^n\big)}^2$. By the proof of Thm. \ref{sxi}, we have 
\begin{align*}
W(\xi) = w.o.t. \text{-}\lim_n w_n.
\end{align*}
Hence, $\mathcal{K}=\{\zeta\in \mathcal{F}_q(\CH): w_n\zeta\overset{w}\rightarrow 0\}$ and it follows that $W(e)\mathcal{K}\subseteq \mathcal{K}$ $($use Lemma \ref{wen}$)$. Thus, decomposing $\mathbf{B}(\mathcal{F}_q(\CH))$ with respect $\mathcal{K}\oplus\mathcal{L}$, one has the form
\begin{align*}
W(e)=\begin{pmatrix}x&y\\0&w\end{pmatrix} \text{ and }u_t=\begin{pmatrix}v_t&0\\0&1\end{pmatrix} \quad \forall\, t. 
\end{align*}
Since $u_tW(e)u_t^* = \sigma_t^\varphi(W(e)) = e^{it\log\lambda} W(e)$ for all $t$ $($use Eq. \eqref{modulartheory}$)$, so choosing $t$ such that $t\log\lambda\not\in2\pi\Z$ it follows that $w=0$. Thus,
$Ran(W(e))\subseteq \mathcal{K}$. 

This is false as $e\not\in\mathcal{K}$. Indeed, 
\begin{align*}
W(\xi)&=w.o.t.\text{-}\lim_n \lambda^{\frac{n}{2}}(1-q)^n W(\bar e)^nW(e)^n\\
&=\lambda^{\frac{1}{2}}(1-q)W(\bar e)W(\xi)W(e).
\end{align*}
Consequently, $\xi=\lambda^{\frac{1}{2}}(1-q)W(\bar e)W(\xi)e\neq 0$. This completes the proof. 
\end{proof}

$$ $$
\noindent{\textbf{Acknowledgements}}: P. Bikram, K. Mukherjee and
\'{E}. Ricard acknowledge the support of the grant 6101-1 from the
CEFIPRA. \'{E}. Ricard is also supported by ANR-19-CE40-0002. S. Wang
would like to thank C. Houdayer and M. Wasilewski for fruitful
discussions at the early stage of this project.

\end{document}